\documentclass[10pt,reqno]{amsart}
\UseRawInputEncoding
\usepackage[utf8]{inputenc}
\usepackage[T1]{fontenc}

\usepackage[usenames, dvipsnames]{color}
\usepackage{ulem}

\usepackage{dsfont, amsfonts, amsmath, amssymb,amscd, stmaryrd, latexsym, amsthm, dsfont}
\usepackage[french,english]{babel}
\usepackage{enumerate}
\usepackage{longtable}
\usepackage{geometry}
\usepackage{float}
\usepackage{tikz}
\usetikzlibrary{shapes,arrows}

\usepackage{float}
\usepackage{tikz}
\usepackage{xypic}
\usetikzlibrary{shapes,arrows}

\usepackage{pifont}
\usepackage{float}
\usepackage{tikz}
\usepackage{xypic}
\usetikzlibrary{shapes,arrows}

\definecolor{darkcerulean}{rgb}{0.03, 0.27, 0.49}
\definecolor{firebrick}{rgb}{0.7, 0.13, 0.13}
\definecolor{forestgreen(traditional)}{rgb}{0.0, 0.27, 0.13}
\definecolor{hanpurple}{rgb}{0.32, 0.09, 0.98}
\definecolor{forestgreen(web)}{rgb}{0.13, 0.55, 0.13}
\newtheorem{theorem}{Theorem}[section]
\newtheorem{lemma}[theorem]{Lemma}
\newtheorem{proposition}[theorem]{Proposition}
\newtheorem{corollary}[theorem]{Corollary}
\newtheorem{remark}[theorem]{\bf Remark}
\newtheorem{conjecture}[theorem]{\bf Conjecture}
\newtheorem{example}[theorem]{\bf Example}

\usepackage[pagebackref]{hyperref}
\renewcommand*{\backref}[1]{}\renewcommand*{\backrefalt}[4]{\ifcase #1 (\tt not cited)\or (\tt cited on page~#2)\else (\tt cited on pages~#2)\fi}

\usepackage{hyperref}
\hypersetup{
	colorlinks=true,
	urlcolor=blue,
	citecolor=blue}
\usepackage{mathrsfs}
\usepackage{amsmath}
\usepackage{amssymb}
\usepackage{amsthm}
\usepackage{amsfonts}
\usepackage{amscd}
\usepackage[mathscr]{eucal}
\usepackage{multirow}
\usepackage[all]{xy}
\newcommand{\Z} {{\mathbb  Z}}
\newcommand{\Q}{{\mathbb  Q}}
\newcommand{\F}{{\mathbb  F}}
\newcommand{\C}{{\mathbb  C}}

\newcommand{\R} {{\mathbb R}}

\def\kk{\mathds{k}}

\begin{document}
	
	\def\NN{\mathbb{N}}
	\def\RR{\mathds{R}}
	\def\HH{I\!\! H}
	\def\QQ{\mathbb{Q}}
	\def\CC{\mathds{C}}
	
	\def\FF{\mathbb{F}}
	\def\KK{\mathbb{K}}
	
	\def\ZZ{\mathbb{Z}}
	\def\DD{\mathds{D}}
	\def\OO{\mathcal{O}}
	\def\kk{\mathds{k}}
	\def\KK{\mathbb{K}}
	
	\def\FF{\mathbb{F}}

	\def\2r{\mathrm{rank_2}}
	\def\rg{\mathrm{rank}}

	\def\vep{\varepsilon}
	\def\Gal{\mathrm{\rm Gal}}	
	\def\G{\mathrm{G}}
	\def\rank{\mathrm{rank}}
	\def\k{\mathrm{k}}
	\def\K{\mathrm{K}}
	\def\L{\mathrm{L}}

	\selectlanguage{english}


\title{Iwasawa invariants and class number parity of multi-quadratic number fields}

	\author[Qinhao Li]{Qinhao Li}
	\address{Qinhao Li: School of Mathematics,
		Nanjing University,
		Nanjing 210093, P.R.China}
	\email{qinhaolimath@qq.com}

	\author[Derong Qiu]{Derong Qiu$^{\ast}$}\thanks{$^{\ast}$ Corresponding author.}
	\address{Derong Qiu: School of Mathematical Sciences,
		Capital Normal University,
		Beijing 100048, P.R.China}
	\email{derong@mail.cnu.edu.cn}

\renewcommand{\subjclassname}{2020 Mathematics Subject Classification}
\subjclass{11R11; 11R27; 11R29}
	\keywords{Iwasawa theory,   Riemann-Hurwitz formula,   class number parity.}


	\begin{abstract}
		In this paper, based mainly on the method of Iwasawa and Kida, by studying
		in detail the Hasse's unit index and the ramifications of prime ideals, we obtain explicit results
		of Iwasawa invariants $ \lambda_{2} $ of the cyclotomic $ \Z_{2}$-extensions of number fields.
		In particular, under Greenberg's conjecture, we obtain an explicit formula of $ \lambda_{2} $
		for imaginary multi-quadratic number fields. As an application, we give a criteria of determining
		class number parity of multi-quadratic number fields.
		
	\end{abstract}

	\selectlanguage{english}
	
	\maketitle

\section{ Introduction}

Let $ K $ be a number field, i.e., a finite extension of $ \Q, $ the field of rational numbers. Let $ p $
be a prime number and let $ \Z_{p} $ denote the ring of $ p$-adic integers. A $ \Z_{p}$-extension of $ K $
is a Galois extension $ K _{\infty} / K $ whose Galois group $ \text{\rm Gal}(K _{\infty} / K) $ is
topologically isomorphic to $ \Z_{p}. $ For each nonnegative integer $ n, $ there exists a unique intermediate
field $ K_{n} $ of $ K _{\infty} / K $ such that $ K_{n} /K $ is a cyclic extension of degree $ p^{n}. $
Moreover, $$ K = K_{0} \subset K_{1} \subset \cdots \subset K_{n} \subset \cdots \subset
K _{\infty} = \bigcup_{n=0}^{\infty} K_{n}. $$
Let $ h(K_{n}) $ denote the class number of $ K_{n}, $ and let $ p^{e_{n}} $ be the exact power of $ p $
dividing $ h(K_{n}). $
A well known spectacular result due to Iwasawa \cite{Iw3},\cite{Wa}, affirms that there exist
integers $ \lambda, \mu \geq 0 $ and $ \nu, $ all independent of $ n, $ and an integer $ n_{0} $ such that
 $$ e_{n} = \lambda n + \mu p^{n} + \nu, $$
for all $ n \geq n_{0}. $  The integers $ \lambda, \mu $ and $ \nu $ are usually called the Iwasawa invariants
of $ K _{\infty} / K. $  \\
$ K $ has a cyclotomic $ \Z_{p}$-extension $ K_{\infty}, $ which is a subfield of $ K(\mu_{p^{\infty}}), $
where $ \mu_{p^{\infty}} $ is the group of all $ p$-power roots of unity in $ \C, $ the field of complex
numbers.

\begin{conjecture}[Iwasawa \cite{Iw4}] $ \mu = 0 $ for any cyclotomic $ \Z_{p}$-extension
$ K_{\infty} / K. $
\end{conjecture}

This was proved by Ferrero and Washington for abelian number field $ K. $
\par \vskip 0.2 cm

\begin{conjecture}[Greenberg \cite{Gre}] \label{Greenberg'sconjecture}
$ \lambda = \mu = 0 $ for any totally real number fields $ K. $  	
	 \end{conjecture}
At present, the two conjectures are generally still open. 	\\
Let $ L $ be a finite Galois $ p$-extension of $ K, $ Kida \cite{Ki1} showed a relation of the Iwasawa
invariants $ \lambda $ and $ \mu $ corresponding to the cyclotomic $ \Z_{p}$-extensions of $ L $ and $ K, $
respectively. This result was later generalized by Iwasawa in \cite{Iw1}, in particular, if $ L / K $ is
a cyclic extension of degree $ p, $ Iwasawa gave a formula for calculating $ \lambda $ of the
cyclotomic $ \Z_{p}$-extensions of $ L $ (see Theorem \ref{IwasawaRiemannHurewitz} in the following).
Such formulas of Iwasawa and Kida are usually viewed as analogous to the Riemann-Hurwitz's formula
of algebraic curves. Schettler \cite{Sc} later generalized it to the case that $ L / K $ is cyclic of $ p$-power degree.
While for $ K = F (\sqrt{d_{1}}, \cdots, \sqrt{d_{r}}, \sqrt{-d}) $ over a real abelian number
field $ F $ of conductor $ \mathfrak{f}_{F/\Q}, $ where $ d, d_{1}, \cdots, d_{r} (\in \Z_{\geq 1}) $
and $ \mathfrak{f}_{F/ \Q} $ are pair-wise prime to each other. By using Iwasawa's formula, Mouhib \cite{Mo1}
explicitly calculated $ \lambda_{2}(K). $  \\
The problem of class number parity has a long history, and has been extensively studied in literature
(e.g., see \cite{CH}, \cite{C}, \cite{F}, \cite{B}, \cite{Ku}, \cite{Mo2}). For real abelian number field
$ F $ of $ 2$-power degree, Fr$\ddot{\text{o}}$hlich proved in \cite[Theorem 5.6]{Fro} that, if there are
more than four prime numbers ramified in $ F / \Q, $ then the class number of $ F $ is odd. Bi-quadratic fields
with odd class number are determined (see e.g. \cite[Theorem 20.3]{CH}, \cite[Lemma 1.3]{C}), and the imaginary
multi-quadratic number fields with class number one are given in \cite{Y} and \cite{F}.  \\
In this paper, firstly we will calculate $ \lambda_{2} $ (i.e., the $ \lambda $ invariant for $ p = 2 $)
in more general cases, then we will use the explicit formula of $ \lambda_{2} $ to consider the problem of
class number parity of multi-quadratic number fields.  \\
For $ \lambda_{2}, $ by studying in detail the Hasse's unit index
and the ramifications of prime ideals, we obtain more explicit results of $ \lambda_{2}. $
In particular, under the Greenberg's conjecture above, we obtain an explicit formula of $ \lambda_{2} $
for imaginary multi-quadratic number fields. As an application, we give a criteria of determining
class number parity of multi-quadratic number fields. Our main results are as follows
\par \vskip 0.2 cm
\begin{theorem}[see Theorem \ref{MainTheorem} and Corollary \ref{MainCorollary} below] Let $ d_{1}, \cdots, d_{r}, d $ be
square-free positive integers, $ F \neq \Q $ be a real abelian number field with conductor
$ \mathfrak{f}_{F/\Q} $ prime to $ d \cdot \prod_{i=1}^{r} d_{i}. $ Let
$ K = F(\sqrt{d_{1}}, \cdots, \sqrt{d_{r}}, \sqrt{-d}) $ be a degree
$ 2^{r+1} $ extension over $ F. $ Assume $ \sqrt{2} \notin K. $ Let $ K_{\infty} $ be the
cyclotomic $ \Z_{2}$-extension of $ K. $
Denote $ K^{(i)} = F(\sqrt{d_{1}}, \cdots, \sqrt{d_{i}}) $ and  $ t_{i} = d_{i} \cdots d_{r} d \ (1 \leq i \leq r). $
Assume there exists an odd prime $ p_{r} $ such that $ p_{r} \mid d_{r} $ but $ p_{r} \nmid d \prod_{i=1}^{r-1} d_{i}. $
For any positive $ n, $ denote $ \text{\rm sqf}(n) = \prod_{\text{prime} \ p|n, 2 \nmid \nu_{p}(n)} p. $
Also, for $ i \in \{1, \cdots, r\}, $ denote  \\
$  s_{i} = \sharp \{v : v \ \text{is a place of} \ K_{\infty}^{(i-1)}(\sqrt{-t_{i}}) \ \text{over a
prime number} \ p \ \text{with} \ p \mid (d_{i}, \text{\rm sqf}(t_{i+1})) \ \text{and} \ p \nmid
2 \prod_{t=1}^{i-1} d_{t} \}, $ \\ and \ $ f_{i} = \sharp \{v : v \ \text{is a place of} \
K_{\infty}^{(i-1)} \ \text{over a prime number} \ p \ \text{with} \ p \mid d_{i} \ \text{and} \ p \nmid
2 \prod_{t=1}^{i-1} d_{t} \}. $  Then
$$  \lambda_{2}(K) = \lambda_{2}(K(\sqrt{2})) = 2^{r} \lambda_{2}(F(\sqrt{-t_{1}})) + \lambda_{2}(K^{+}) - 2^{r} \lambda_{2}(F)
+ \sum_{i=1}^{r} 2^{r-i} (s_{i} - f_{i}) + \delta,  $$
where $ \delta = \left \{\begin{array}{l} 1, \quad \quad \text{if} \ \sqrt{d} \in K^{+}(\sqrt{2}),  \\
0, \quad \quad  \text{if otherwise}.
\end{array} \right. $   \\
If $ F = \Q, $ i.e., $ K = \Q(\sqrt{d_{1}}, \cdots, \sqrt{d_{r}}, \sqrt{-d}) $ with
$ [K : \Q] = 2^{r+1} $ as above, then
$$  \lambda_{2}(K) = \lambda_{2}(K^{+}) + \sum_{\substack{\text{primes} \ p | \prod_{i=1}^{r} d_i, \\ p \neq 2}}
2^{\nu_{2}(p^{2} - 1) + r - \theta - 4} + \sum_{\substack{\text{primes} \ p | d, \\ p\nmid 2 \prod_{i=1}^{r} d_{i}}}
2^{\nu_{2}(p^{2} - 1) + r - \theta - 3} - 2^{r - \theta} + \delta.  $$
where $ \theta = \left \{\begin{array}{l} 1, \quad \text{if} \ \sqrt{2} \in K^{+},  \\
0, \quad \text{if otherwise}.
\end{array} \right. $
\end{theorem}

\begin{remark}
 In particular, if Greenberg's conjecture holds for $ K^{+}, $ then
for case $ F = \Q, $ we have the following explicit formula
$$  \lambda_{2}(K) = \sum_{\substack{\text{primes} \ p | \prod_{i=1}^{r} d_i, \\ p \neq 2}}
2^{\nu_{2}(p^{2} - 1) + r - \theta - 4} + \sum_{\substack{\text{primes} \ p | d, \\ p\nmid 2 \prod_{i=1}^{r} d_{i}}}
2^{\nu_{2}(p^{2} - 1) + r - \theta - 3} - 2^{r - \theta} + \delta.  $$
\end{remark}

\begin{theorem}[see Theorem \ref{MainApplicationParity} below] Let $ K $ be an imaginary multi-quadratic number field
with $ \sqrt{2} \in K. $ Then $ 2 \nmid h(K) $ if and only if $ K $ is taken as one of the following form
\begin{enumerate}[$1)$]
\item $ \Q(\sqrt{2}, \sqrt{-p}), $ where $ p $ is a prime number with $ p \equiv 3 \ (\text{\rm mod} \ 8); $
\item $ \Q(\sqrt{2}, \sqrt{-1}, \sqrt{-p}), $ where $ p $ is a prime number with $ p \equiv 3, 5 \ (\text{\rm mod} \ 8); $			
\item $ \Q(\sqrt{2}, \sqrt{-p}, \sqrt{-q}), $ where $ p $ and $ q $ are different prime numbers with
$ p, q \equiv 3 \ (\text{\rm mod} \ 8); $
\item $ \Q(\sqrt{2}, \sqrt{-1}). $
		\end{enumerate}
Moreover, for $ K = \Q(\sqrt{2}, \sqrt{-p}), $ if $ p \equiv 5 \ (\text{\rm mod} \ 8), $ then $ 2 \mid h(K) $
but $ 4 \nmid h(K). $
\end{theorem}
\par     \vskip  0.4 cm

\section{ Preliminary}   

\begin{lemma}\label{LiftingTheExponent}
Let $ p $ be a prime number, $ K $ a finite extension of $ \Q_{p} $
with normalized valuation $ v_{p} $ such that $ v_{p}(p) = 1. $ Then for any
$ x, y \in K^{\times} = K \setminus \{0\}, $ if $ v_{p}(x-y) > \frac{1}{p-1}, $ we have
$$ v_{p}(x^{n} - y^{n}) = (n-1)v_{p}(x) + v_{p}(x - y) + v_{p}(n) \quad (\forall
n \in \Z_{\geq 1}). $$
\end{lemma}

\begin{proof} By assumption, $ xy \neq 0. $ Since $ v_{p}(x^{n} - y^{n}) =
v_{p}((\frac{x}{y})^{n} - 1) + v_{p}(y^{n}), $ we only need to show the equality on case
$ y = 1. $ For this, firstly, let $ n = p^{r} \ (r \in \Z_{\geq 1}) $ be a $ p$-power. We use
induction on the exponent $ r. $ For $ r = 1, $ i.e., $ n = p, $ we write $ z = x-y=x-1, $
then $ v_{p}(x^{p} - 1) = v_{p}\left(\sum_{i=1}^{p}\begin{pmatrix}  p \\
i \end{pmatrix} z^{i}\right). $ Notice that \\
$ v_{p}\left(\begin{pmatrix}  p \\
i \end{pmatrix} z^{i}\right) =
\begin{cases} v_{p}(z) + 1,  & i = 1,  \\
\geq i v_{p}(z) + 1,  & 2\leq i \leq p-1,  \\
p v_{p}(z), &i = p.
\end{cases} $
By assumption, $ v_{p}(z) > \frac{1}{p-1} > 0. $
So $ p v_{p}(z) > v_{p}(z) + 1 $ and $ i v_{p}(z) + 1 > v_{p}(z) + 1 $ for all
$ 2 \leq i \leq p - 1. $ So the equality holds.
Suppose that the equality holds for $ r. $ Then for $ r + 1, $ we have
$ v_{p}(x^{p^{r+1}} - 1) = v_{p}((x^{p})^{p^{r}} - 1) = (p^{r} - 1)  v_{p}(x^{p})
+ v_{p}(x^{p} - 1) + v_{p}(p^{r}) = p(p^{r} - 1) v_{p}(x) +  v_{p}(x^{p} - 1) + r =
(p^{r+1} - 1) v_{p}(x) + v_{p}(x - 1) + v_{p}(p^{r + 1}). $ So the equality holds for
any $ p$-power $ n. $ Secondly, suppose $ p \nmid n. $ Notice that
$ v_{p}\left(\begin{pmatrix}  n \\
i \end{pmatrix} z^{i}\right) =
\begin{cases} v_{p}(z),& i = 1,  \\
\geq i v_{p}(z),&  2\leq i \leq n.
\end{cases} $
So again by the assumption $ v_{p}(z) > \frac{1}{p-1} > 0, $ we get $ v_{p}(x^{n} - 1) =
v_{p}\left(\sum_{i=1}^{n}\begin{pmatrix}  n \\
i \end{pmatrix} z^{i}\right) = v_{p}(z). $
The equality holds. Lastly, for general case, we may write $ n = p^{r} m,  p \nmid m. $
Then by above discussion, we have $  v_{p}(x^{n} - 1) =  v_{p}((x^{p^{r}})^{m} - 1) =
(m-1) v_{p}(x^{p^{r}}) + v_{p}(x^{p^{r}} - 1) + v_{p}(m) = (m-1) p^{r}v_{p}(x) + (p^{r} - 1) v_{p}(x)
+ v_{p}(x - 1) + v_{p}(p^{r}) = (n - 1) v_{p}(x) + v_{p}(x - 1) + v_{p}(n). $
The proof is completed.
\end{proof}
\begin{lemma}  Let $ p $ be an odd prime number, $ n $ a positive integer.
Let $ r_{2^{n}}(p) $ be the order of the residue class of $ p $ in the multiplicative group
$ (\Z/2^{n} \Z)^{\times}. $ Then
$$
 r_{2^{n}}(p) =  \begin{cases} 1, &
\text{if} \ n \leq v_{2}(p-1),  \\
2, &\text{if} \ v_{2}(p-1) < n \leq v_{2}(p^{2}-1),  \\
2^{n- v_{2}(p^{2}-1)+1}, & \text{if} \ n > v_{2}(p^{2}-1).
\end{cases}
$$
\end{lemma}

\begin{proof} If $ n \leq v_{2}(p-1), $ then $ p \equiv 1 (\text{\rm mod} \ 2^{n}), $ so
$ r_{2^{n}}(p) = 1. $ If $ v_{2}(p-1) < n \leq v_{2}(p^{2}-1), $ then $ p \not\equiv 1 (\text{\rm mod} \ 2^{n}) $
and $ p^{2} \equiv 1 (\text{\rm mod} \ 2^{n}), $ so $ r_{2^{n}}(p) = 2. $ Lastly, if
$ n > v_{2}(p^{2}-1), $ let $ m = n + 1 - v_{2}(p^{2}-1), $ then by taking $ x = p^{2}, y = 1 $
in Lemma \ref{LiftingTheExponent} above, we get $ v_{2}(p^{2^{m}}-1) = v_{2}((p^{2})^{2^{m-1}}-1) =
(2^{m-1}-1)v_{2}(p^{2}) + v_{2}(p^{2}-1) + v_{2}(2^{m-1}) = v_{2}(p^{2}-1) + m - 1 = n. $
So $ p^{2^{m}} \equiv 1 (\text{\rm mod} \ 2^{n}). $ On the other hand, for any positive integer
$ m^{\prime } < m, $ we have $ v_{2}(p^{2^{m^{\prime}}}-1) = v_{2}(p^{2}-1) + m^{\prime } - 1
< n. $ Hence $ p^{2^{m^{\prime}}} \not\equiv 1 (\text{\rm mod} \ 2^{n}), $ this shows that
$ r_{2^{n}}(p) = 2^{m} = 2^{n- v_{2}(p^{2}-1)+1}, $ and the proof is completed.
 \end{proof}
\begin{lemma}  Let $ p $ be an odd prime number, $ n $ a positive integer. If
$ p^{t} \equiv -1 (\text{\rm mod} \ 2^{n}) $ for some integer $ t \in \Z, $ then
$ p \equiv -1 (\text{\rm mod} \ 2^{n}). $
\end{lemma}

\begin{proof}  The case $ n = 1 $ or $ 2 $ is obvious. Now assume $ n \geq 3. $
Then $ (\Z/2^{n}\Z)^{\times} = \left< -1 \right> \times \left< 5 \right> \cong  \Z/2\Z \times \Z/2^{n-2} \Z. $
We have $ p \equiv (-1)^{a} 5^{b} (\text{\rm mod} \ 2^{n}) $ for some $ a \in \{0, 1 \} $
and $ b \in \{0, 1, \cdots, 2^{n-2} - 1\}. $ By assumption,
$ p^{t} \equiv -1 (\text{\rm mod} \ 2^{n}) $ holds for some integer $ t \in \Z, $ so
$ (-1)^{at} 5^{bt} \equiv -1 (\text{\rm mod} \ 2^{n}). $ If $ b \neq 0, $ then $ 2 \nmid t, $
and $ 2^{n-2} \mid bt, $ so $ 2^{n-2} \mid b. $ But $ v_{2}(bt) = v_{2}(b) < n-2. $
A contradiction! Hence $ b = 0, $ so by $ (-1)^{at} 5^{bt} \equiv -1 (\text{\rm mod} \ 2^{n}), $
we have $ a = 1. $ Therefore $ p \equiv (-1)^{a} \equiv -1 (\text{\rm mod} \ 2^{n}). $
The proof is completed.
\end{proof}
\begin{lemma} [\cite{Ma}, Chapter 4, Exercise 12, p.83] Let $ m \in \Z_{\geq 1}, $ and
$ \zeta _{m} = e^{2\pi i/m} $ be a $ m$-th root of unity. Let $ G =
\text{\rm Gal}(\Q(\zeta _{m})/ \Q) \cong (\Z / m \Z)^{\times } $
be the Galois group of the cyclotomic extension $ \Q(\zeta _{m})/ \Q. $ Let $ K \leq
\Q(\zeta _{m}) $ be a subfield, and $ H = \text{\rm Gal}(\Q(\zeta _{m})/ K). $ We view
$ H $ as a subgroup of $ (\Z / m \Z)^{\times }. $ Let $ p \nmid m $ be a prime number,
$ \mathfrak{p} $ and $ \mathfrak{P} $ be any places of $ K $ and $ \Q(\zeta _{m}), $
respectively, such that $ \mathfrak{P} | \mathfrak{p} | p. $ Then the degree of residual class
$$ f(\mathfrak{p}/p) = \min \{f \in \Z_{\geq 1} : \ p^{f} \text{\rm mod} \ m \in H \}. $$
\end{lemma}
\begin{proof} By assumption, $ \mathfrak{p} $ is unramified in $ K/\Q. $ By definition,
$ \phi (\zeta _{m}) = \zeta _{m}^{p} \text{\rm mod} \mathfrak{P}, $ where $ \phi =
\left(\frac{\Q(\zeta _{m})/ \Q}{\mathfrak{p}} \right) $ is the Frobenius element. So
$ \phi \equiv p \in (\Z / m \Z)^{\times }. $ The decomposition
group $ D_{\mathfrak{P}} = \left<\phi \right> $ acts on the quotient space $ H \backslash G $ of
right cosets via $ (Hg) \ast \phi = H(g \phi) \ (\forall g \in G). $ Write $ f_{0} =
\min \{f \in \Z_{\geq 1} : \ p^{f} \text{\rm mod} \ m \in H \}. $ For any $ g \in G, $ let
$ (Hg) \ast \phi ^{k} = Hg, $ i.e., $ (Hg)p^{k} = Hg. $
So the stabilized subgroup $ \text{Stab}_{D_{\mathfrak{P}}}(Hg) = \left<f_{0} \right> \subset
\left< p \right>. $
Hence $ \sharp D_{\mathfrak{P}} / \sharp \text{Stab}_{D_{\mathfrak{P}}}(Hg) = f_{0}. $
Then by \cite[Theorem 33]{Ma}, $ f(\mathfrak{p}/p) = f_{0}. $ The proof is completed.
\end{proof}

Let $ \Q_{\infty} = \bigcup _{n \geq 0} \Q_{n} $ be the cyclotomic $ \Z_{2}$-extension
of $ \Q $ with Galois group $ \text{\rm Gal}(\Q_{\infty} / \Q) \cong \Z_{2}, $ the additive
group of $ 2$-adic integers, where, for $ n \geq 0, \Q = \Q_{0} \subset \Q_{1} \subset
\cdots \subset \Q_{n} \subset \cdots. $ It is not difficult to see that $ \Q_{n} =
\Q (\zeta_{2^{n+2}} + \zeta_{2^{n+2}}^{-1}). $ Let $ h_{n} $ denote the class number of
$ \Q_{n}, 2^{e_{n}} $ the exact power of $ 2 $ dividing $ h_{n}. $ Then by Iwasawa's theorem,
there exist integers $ \lambda_{2}, \mu_{2} $ and $ \nu_{2}, $ depending only on
$ \Q_{\infty} / \Q, $ such that $ e_{n} = \lambda_{2} n + \mu_{2} \cdot 2^{n} + \nu_{2} $
for $ n \gg 0. $
\begin{corollary} Let $ p $ be an odd prime number. Denote $ f_{2^{n}}(p) =
\min \{f \in \Z_{\geq 1} : p^{f} \equiv \pm 1 (\text{\rm mod} \ 2^{n})\}. $ Then \
$ f_{2^{n}}(p) =  \begin{cases} 1,& \text{if} \ n < v_{2}(p^{2}-1),  \\
2^{n- v_{2}(p^{2}-1)+1}, & \text{if} \ n \geq v_{2}(p^{2}-1).
\end{cases} $  \\
Moreover, the residue degree of $ p $ in the extension $ \Q_{n} / \Q $ is equal to
$ f_{2^{n+2}}(p). $
\end{corollary}

\begin{proof} Since $ \text{\rm Gal}(\Q (\zeta_{2^{n+2}}) / \Q_{n}) = \left<-1 \right> \subset
(\Z/2^{n+2} \Z)^{\times} \cong \text{\rm Gal}(\Q (\zeta_{2^{n+2}}) / \Q) $ and
$ v_{2}(p^{2} - 1) - \max \{v_{2}(p - 1), v_{2}(p + 1)\} = 1, $ the conclusion follows
from the above Lemmas.
\end{proof}
\begin{corollary}\label{DecompositionsOfOddPrimeInQn}
Let $ p $ be an odd prime number. Denote $$ g_{n}(p) =
\sharp \{\mathfrak{p} : \mathfrak{p} \ \text{is a prime ideal of} \ \Q_{n} \
\text{and} \ \mathfrak{p} | p \}. $$ Then \
$ g_{n}(p) =  \begin{cases} 2^{n},  & \text{if} \ n \leq v_{2}(p^{2}-1) - 3,  \\
2^{v_{2}(p^{2}-1)-3}, & \text{if} \ n > v_{2}(p^{2}-1) - 3.
\end{cases}  $
\end{corollary}

\begin{proof} Obvious.
\end{proof}

\begin{proposition}\label{SplitBiQuadra}
Let $ p $ be an odd prime number, $ F $ a number field, $ K $ and $ K^{\prime } $
be any two different quadratic extensions of $ F. $ If every prime of $ F $ above $ p $ is inert in both
$ K $ and $ K^{\prime }, $ then all the primes above $ p $ split completely in $ KK^{\prime } / K. $
\end{proposition}
\begin{proof} For any prime ideal $ \mathfrak{p} $ in $ F $ above $ p, $ by assumption,
$ \mathfrak{p} $ is inert in both $ K $ and $ K^{\prime }. $ So $ \mathfrak{p} $ is unramified in
$ KK^{\prime}. $ It is easy to see that the Frobenius element
$ \left(\frac{KK^{\prime } / F}{\mathfrak{p}} \right) \in \text{\rm Gal}(K/F) \times \text{\rm Gal}(K^{\prime}/F) $
is of order $ 2. $ So the decomposition group of $ \mathfrak{p} $ in $ KK^{\prime } / F $ consists
of only two elements, this shows that there are only two different prime ideals in $  KK^{\prime } $
above $ \mathfrak{p}, $ and we are done.
\end{proof}

\begin{corollary} Let $ p $ be an odd prime number, and $ d $ a square-free integer with
$ p \nmid d. $
\begin{enumerate}[$1)$]
\item For any integer $ n > v_{2}(p^{2} - 1) - 3, $ all the prime ideals above $ p $ split completely
in $ \Q_{n}(\sqrt{d})/ \Q_{n}. $			
\item If $ \left(\frac{d}{p} \right) = 1, $ then for any positive integer $ n, $ all the prime ideals above
$ p $ split completely in $ \Q_{n}(\sqrt{d})/ \Q_{n}. $ 			
\item If $ \left(\frac{d}{p} \right) = -1, $ then for any positive integer $ n \leq v_{2}(p^{2} - 1) - 3, $
all the prime ideals above $ p $ are inert in $ \Q_{n}(\sqrt{d})/ \Q_{n}. $
		\end{enumerate}
\end{corollary}

\begin{proof} Follows directly from \cite[Chapter II, Exercise 1.2.5]{Gra},
Corollary \ref{DecompositionsOfOddPrimeInQn} and Proposition \ref{SplitBiQuadra}.
\end{proof}

\begin{corollary}\label{QuadraticExtofQInfiniteDecom}
Let $ p $ be an odd prime number, $ d_{1}, \cdots, d_{r} $ are
square-free integers, $ r \geq 1. $
\begin{enumerate}[$1)$]
\item All the prime ideals above $ p $ split completely
in $ \Q_{\infty }(\sqrt{d_{1}}, \sqrt{d_{2}}, \cdots, \sqrt{d_{r}})/
\Q_{\infty }(\sqrt{d_{2}}, \cdots, \sqrt{d_{r}}) $ if $ p \nmid d_{1}. $
\item All the prime ideals above $ p $ split completely
in $ \Q_{\infty }(\sqrt{d_{1}}, \sqrt{d_{2}}, \cdots, \sqrt{d_{r-1}})/\Q_{\infty } $
if $ p \nmid \prod_{i=1}^{r-1} d_{i}. $
\end{enumerate}
\end{corollary}

\begin{proof} By induction on $ r, $ it follows directly from \cite[Chapter II, Exercise 1.2.5]{Gra}.
\end{proof}

\begin{lemma}\label{Lemma 2.10} Let $ F $ be a real abelian number field,
$ d, d_{1}, \cdots, d_{r} $ are square-free integers. $ r \geq 1. $ Let
$ F(\sqrt{d_{1}}, \cdots, \sqrt{d_{r}}, \sqrt{d}) /F(\sqrt{d_{1}}, \cdots, \sqrt{d_{r}}) $
be an extension of degree $ 2^{r+1}. $ If the conductor $ \mathfrak{f}_{F/\Q} $ of $ F/\Q $
is prime to $ d \prod_{i=1}^{r} d_{i}, $ then for a finite place $ \mathfrak{p} \nmid 2 $ in
$ F(\sqrt{d_{1}}, \cdots, \sqrt{d_{r}}), \mathfrak{p} $ ramifies in
$ F(\sqrt{d_{1}}, \cdots, \sqrt{d_{r}}, \sqrt{d}) $ if and only
if $ \mathfrak{p} | d $ and $ \mathfrak{p} \nmid \prod_{i=1}^{r} d_{i}. $ Moreover,
for such $ \mathfrak{p}, $ let $ p $ be the rational prime number such that
$ \mathfrak{p} | p, $ then all the places over $ p $ split completely in
$  F_{\infty}(\sqrt{d_{1}}, \cdots, \sqrt{d_{r}}) / F_{\infty}, $ where $ F_{\infty} $
is the cyclotomic $ \Z_{2}$-extension.
\end{lemma}

\begin{proof} If $ \mathfrak{p} \nmid 2 $ ramifies in
$ F(\sqrt{d_{1}}, \cdots, \sqrt{d_{r}}, \sqrt{d}), $
then by the theory of Kummer extension, $ \mathfrak{p} | d. $ Let $ \mathfrak{p} \bigcap
\mathcal{O}_{F} = \mathfrak{p}_{F}. $ If $ \mathfrak{p}_{F} | d_{i} $ for some
$ i \in \{1, \cdots, r\}, $ then $ \mathfrak{p}_{F} $ is unramified in
$ F(\sqrt{dd_{i}}) / F, $ so all the primes over $ \mathfrak{p}_{F} $ are unramified in
$ F(\sqrt{d_{i}}, \sqrt{d}) / F(\sqrt{d_{i}}), $ which implies that $ \mathfrak{p} $
is unramified in $ F(\sqrt{d_{1}}, \cdots, \sqrt{d_{r}}, \sqrt{d}). $ Conversely,
if $ \mathfrak{p} \nmid 2 \prod_{i=1}^{r} d_{i}, $ then $ \mathfrak{p} $ is unramified
in $ F(\sqrt{d_{1}}, \cdots, \sqrt{d_{r}})/ \Q $ because $ \mathfrak{f}_{F/\Q} $ is
prime to $ d \prod_{i=1}^{r} d_{i}. $ So for $ p $ under $ \mathfrak{p}, $ if $ p $
ramifies in $ \Q(\sqrt{d}) / \Q, $ then $ \mathfrak{p} $ also ramifies in
$ F(\sqrt{d_{1}}, \cdots, \sqrt{d_{r}}, \sqrt{d}). $ Lastly, for such $ p, $ by 2) of
Corollary \ref{QuadraticExtofQInfiniteDecom} above, all primes over $ p $ split completely in
$ \Q_{\infty}(\sqrt{d_{1}}, \cdots, \sqrt{d_{r}}) / \Q_{\infty}, $ hence by
\cite[Chapter II, Exercise 1.2.5]{Gra}, all the places over $ p $ split completely in
$  F_{\infty}(\sqrt{d_{1}}, \cdots, \sqrt{d_{r}}) / F_{\infty}. $
The proof is completed.
\end{proof}

\section{ Iwasawa $ \lambda _{2}$-invariant and Riemann-Hurwitz formula}  

Let $ F $ be a number field, we denote by $ E_{F}, W_{F} $ the group of units, the group of roots of unity of $ F, $
respectively. For a prime number $ p, $ let $ F_{\infty} $ be a $ \Z_{p}$-extension of $ F, $ we denote
$ E_{F_{\infty}} = \bigcup_{n = 1}^{\infty} E_{F_{n}} $ and $ W_{F_{\infty}} = \bigcup_{n = 1}^{\infty} W_{F_{n}}. $
For $ n \in \Z_{\geq 1}, \zeta _{n} = e^{2\pi i/n} $ is a primitive $ n$-th root of unity. For an abelian group
$ A $ and a prime number $ q, $ let $ A(q) $ denote the subgroup of all elements of $ A $ whose order is a power
of $ q. $ \\
Let $ K/k $ be a cyclic extension of number fields with Galois group
$ G = \text{\rm Gal}(K/k) \cong \Z / p \Z, $ and $ K_{\infty}, k_{\infty} $ be the cyclotomic
$ \Z_{p}$-extensions of $ K, k $ respectively.

\begin{theorem}[Iwasawa \cite{Iw1}]\label{IwasawaRiemannHurewitz}
Let $ K_{\infty}/k_{\infty} $ be a cyclic extension of degree $ p, $
and $ K_{\infty}/k_{\infty} $ are unramified at all the infinite places. Then
$$ \lambda_{p}(K) = p \lambda_{p}(k) + \sum_{w \nmid p}
(e(w /v) - 1) + (p - 1) \cdot \chi (G, E_{K_{\infty}}), $$
the sum being taken over all places $ w $ in $ K_{\infty} $ which does not divide $ p, $
and $ v = w \mid_{k_{\infty}}. $ Here for the $ G$-module $ E_{K_{\infty}}, \
\widehat{H}^{i}(G, E_{K_{\infty}}) $ is the $ i$th Tate cohomology, and $ \chi (G, E_{K_{\infty}}) =
\dim_{\F_{p}} \widehat{H}^{0}(G, E_{K_{\infty}}) - \dim_{\F_{p}} \widehat{H}^{1}(G, E_{K_{\infty}}). $
\end{theorem}

\begin{theorem}[See \cite{Fe}, \cite{Ki2}]\label{FerreroImaQuaLambda}
Let $ d > 0 $ be a square-free integer, $ K = \Q(\sqrt{-d}). $
Then
$$\lambda_{2}(K) =  \begin{cases} \sum_{\text{primes} \ p |d, p \not= 2} 2^{\nu_2(p^2-1)-3}-1,
& \text{if} \ d \notin \{1, 2 \},   \\
0,  & \text{if otherwise.}
\end{cases}   $$
\end{theorem}

\begin{theorem}[Hasse, see \cite{Wa}] For any CM number field
$ K, \ (E_{K} : E_{K^{+}} W_{K}) = 1 $ or $ 2. $
\end{theorem}

Now let $ F $ be a real abelian number field, $ d, d_{1}, \cdots, d_{r} $ be
square-free positive integers, and let $ K = F(\sqrt{d_{1}}, \cdots, \sqrt{d_{r}}, \sqrt{-d}) $
be a degree $ 2^{r+1} $ extension of $ F. $ Then $ K $ is a CM field, and its maximal real
subfield $ K^{+} = F(\sqrt{d_{1}}, \cdots, \sqrt{d_{r}}). $ Let $ K_{\infty} =
\bigcup_{n=0}^{\infty} K_{n} $ be the cyclotomic $ \Z_{2}$-extension, $ K_{0} = K $ and
$ [K_{n} : K] = 2^{n}. $ Denote $ W_{K_{n}} = W_{K_{n}}(2) \times W_{K_n}^{\prime }, $
where $ W_{K_n}^{\prime } $ is the subgroup of all elements of $ W_{K_{n}} $ whose order is odd.   \\
Obviously, $ K(\sqrt{2}) $ and $ K $ have the same cyclotomic $ \Z_{2}$-extension, so we may as well
assume that all the integers $ d, d_{1}, \cdots, d_{r} $ are odd and $ \sqrt{2} \notin K. $
Then $ K \bigcap \Q_{\infty} = \Q. $

\begin{lemma}\label{RootsOfUnityinKInfty}
For $ K $ and $ K_{n} \ (n \geq 1) $ as above, we have
$$ W_{K_{n}}(2) =  \begin{cases} \{\pm 1\}, &
\text{if} \ \sqrt{d} \notin K^{+}(\sqrt{2}),  \\
\left<\zeta_{2n+2} \right>, &  \text{if} \ \sqrt{d} \in K^{+}(\sqrt{2}).
\end{cases}  $$
\end{lemma}

\begin{proof} Since the extension $ K_{n}^{+}/ K^{+} $ is cyclic and $ 2 \nmid d, $
it is easy to see that, $ \sqrt{-1} \in K_{n} \Leftrightarrow \sqrt{d} \in K_{n}^{+}
\Leftrightarrow \sqrt{d} \in K_{1} = K(\sqrt{2}). $
If $ \sqrt{-1} \in K_{n}, $ since $ \Q(\zeta_{2^{n+2}}) =
\Q(\zeta_{2^{{n+2}}} + \zeta_{2^{{n+2}}}^{-1})(\sqrt{-1}), $ we have
$ K_{n} = \Q_{n} K = \Q(\zeta_{2^{{n+2}}} + \zeta_{2^{{n+2}}}^{-1}) K = \Q(\zeta_{2^{n+2}}) K.
$ So $ \zeta_{2^{n+2}} \in K_{n}. $ Conversely, if $ \zeta_{2^{n+2}} \in K_{n}, $ then
$ \sqrt{-1} \in \Q(\zeta_{2^{{n+2}}}) \subset K_{n}. $ So $ \sqrt{-1} \in K_{n}
\Leftrightarrow \zeta_{2^{{n+2}}} \in K_{n}. $ If $  \zeta_{2^{{n+3}}} \in K_{n}, $
then $ K^{+}(\zeta_{2^{n+3}} + \zeta_{2^{n+3}}^{-1}) \subset K_{n}^{+}. $ However,
$ [K^{+}(\zeta_{2^{n+3}} + \zeta_{2^{n+3}}^{-1}) : K^{+}] = 2^{n+1}, $ a contradiction.
Therefore, $ W_{K_{n}}(2) = \left<\zeta_{2n+2} \right>. $ The proof is completed.
\end{proof}

\begin{proposition}\label{HasseUnitIndex}
Let $ K = F(\sqrt{d_{1}}, \cdots, \sqrt{d_{r}}, \sqrt{-d}) $
be a CM field as above, where $ d_{1}, \cdots, d_{r}, d \in \Z_{> 0} $ are odd square-free
integers. Assume $ \sqrt{2} \notin K^{+} $ and $ [K : F] = 2^{r+1}. $ If one of the
following conditions holds:
\begin{enumerate}[$1)$]
\item $ \sqrt{d} \in K^{+}(\sqrt{2}); $
\item There exists an odd prime $ p $ such that $ p \mid d $ and
$ p \nmid \mathfrak{f}_{F/\Q} \cdot \prod_{i=1}^{r} d_{i}, $ where
$ \mathfrak{f}_{F/\Q} $ is the conductor of $ F. $
\end{enumerate}
Then $ (E_{K_{\infty}} : E_{K_{\infty}^{+}}W_{K_{\infty}}) = 1. $
\end{proposition}

\begin{proof} 1) \ Assume $ \sqrt{d}\in K^{+}(\sqrt{2}). $ Then $ (E_{K_{\infty}} : E_{K_{\infty}^{+}} W_{K_{\infty}}) = 1. $
If otherwise, then $( E_{K_{\infty}} : E_{K_{\infty}^{+}} W_{K_{\infty}}) = 2, $ so there exists a unit
$ \varepsilon \in E_{K_{\infty}} $ such that $ \varepsilon \notin E_{K_{\infty}^{+}} W_{K_{\infty}} $ and
$ \varepsilon^{2} \in E_{K_{\infty}^{+}} W_{K_{\infty}}. $ Let $ \varepsilon^{2} = \nu \zeta, $ where
$ \zeta \in W_{K} $ and $ \nu \in E_{K_{\infty}^{+}}. $ We may as well assume that $ \nu > 0. $ Write
$ \zeta = \zeta_{2^{t}} \zeta_{2m+1}, $ where $ \zeta_{2^{t}} \in W_{K_{\infty}}(2) $ and $ \zeta_{2m+1} \in
W_{K_{\infty}}^{\prime }. $ Then $ \varepsilon^{2} = \nu(\zeta_{2^{t+1}})^{2} (\zeta_{2m+1}^{m+1})^{2}. $
By Lemma \ref{RootsOfUnityinKInfty} above we know that $ W_{K_{\infty}}(2) $ contains all the $ 2$-power roots of unity.
So $ \sqrt{\nu} = \pm \varepsilon / (\zeta_{2^{t+1}} \zeta_{2m+1}^{m+1}) \in E_{K_{\infty}}, $
hence $ \sqrt{\nu} \in E_{K^{+}}, $ and so $ \varepsilon = \pm \sqrt{\nu} (\zeta_{2^{t+1}}) (\zeta_{2^{2m+1}}^{m+1})
\in E_{K_{\infty}^{+}} W_{K_{\infty}}, $ a contradiction. \\
2) \ If otherwise, then $ (E_{K_{n}} : E_{K_{n}^{+}} W_{K_{n}}) = 2, $ so there exists a unit
$ \varepsilon \in E_{K_{n}} $ such that $ \varepsilon^{2} \in E_{K_{n}^{+}} W_{K_n} $ and
$ \varepsilon \notin E_{K_{n}^{+}} W_{K_n}. $ Let $ \varepsilon^{2} = \zeta_{2m+1} v, $ where
$ \zeta_{2m+1} \in W_{K_{n}} $ and $ v \in E_{K_{n}^{+}}. $ Notice that $ v < 0. $ In fact, if
$ v \geq 0, $ then $ \sqrt{v} \in \R \cap E_{K_{n}} = E_{K_{n}^{+}}, $ and $ \zeta_{2(2m+1)} =
\pm \varepsilon / \sqrt{v} \in E_{K_{n}}, $ so $ \zeta_{2(2m+1)} \in W_{K_{n}}, $ hence $ \varepsilon =
\pm \zeta_{2(2m+1)} v \in W_{K_{n}} E_{K_{n}^{+}}, $ which contracts to the choice of $ \varepsilon. $
Thus $ v < 0, $ then for $ K_{n} = K_{n}^{+} (\sqrt{v}), $ the extension $ K_{n} /  K_{n}^{+} $ is
unramified outside $ 2. $ Moreover, as a sub-extension of the $ \Z_{2}$-extension, $ K_{n}^{+} / K^{+} $
is unramified outside $ 2. $ Hence $ K_{n} / K^{+}$ is unramified outside $ 2. $
$$     
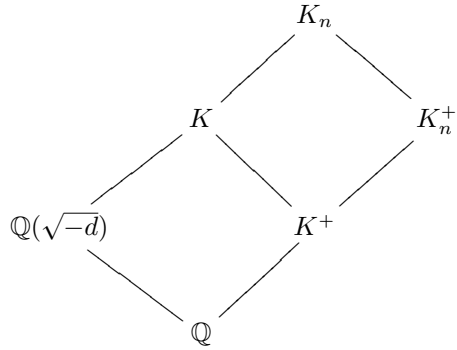
\xymatrix@C=1.51em@R=2em{
        & & K_{n}& \\
        & K\ar@{-}[ur]& &K_{n}^{+}\ar@{-}[ul] \\
        \mathbb{Q}(\sqrt{-d})\ar@{-}[ur]& & K^{+}\ar@{-}[ur]\ar@{-}[ul]& \\
        & \mathbb{Q}\ar@{-}[ur]\ar@{-}[ul]& &
    }   $$
If there exists an odd prime $ p $ such that $ p \mid d $ and
$ p \nmid \mathfrak{f}_{F/\Q} \cdot \prod_{i=1}^{r} d_{i}, $ then $ p $ is ramified in $ \Q(\sqrt{-d})/\Q $
but unramified in $ K^{+}/\Q, $ hence all the primes over $ p $ are ramified in $ K / K^{+}, $
a contradiction. The proof is completed.
\end{proof}

\begin{remark}
     The result 2) of Proposition \ref{HasseUnitIndex} above can also be deduced from of \cite[Theorem 1]{Le}.
\end{remark}

Now for $ d_{1}, \cdots, d_{r}, d \in \Z_{> 0} $ as above, we denote \\ $ K^{(0)} = F, \ K^{(i)} =
F(\sqrt{d_{1}}, \cdots, \sqrt{d_{i}}), t_{i} = d_{i} \cdots d_{r} d, \ G_{i} =
\text{\rm Gal}(K^{(i)} / K^{(i-1)} ) \ (1 \leq i \leq r), $ and write $ t_{r+1} = d. $ Then
$$ G_{i} \cong \text{\rm Gal}(K_{\infty}^{(i)} / K_{\infty}^{(i-1)}) \cong
\text{\rm Gal}(K_{\infty}^{(i)}(\sqrt{-t_{i+1}})/K^{(i-1)}_{\infty}(\sqrt{-t_{i}})). $$
All $ E_{K_{\infty}^{(i)}(\sqrt{-t_{i}})} $ and $ E_{K_{\infty}^{(i)}} $ are $ G_{i}$-modules
($ 1 \leq i \leq r$).
$$
\xymatrix@C=58pt{
    & K  & & K_{\infty}\\
    K^{(r)}\ar@{-}[ur] & & K^{(r)}_{\infty}\ar@{-}[ur] & \\
    &K^{(r-1)}(\sqrt{-t_r})\ar@{-}[uu]_{G_r} & &K^{(r-1)}_{\infty}(\sqrt{-t_r})\ar@{-}[uu]_{G_r} \\
    K^{(r-1)}\ar@{-}[ur]\ar@{-}[uu]^{Gr} & & K^{(r-1)}_{\infty}\ar@{-}[ur]\ar@{-}[uu]^{G_r} & \\
     & K^{(1)}(\sqrt{-t_2})\ar@{.}[uu]\ar@{=>}[r]^{\text{\qquad \text{cyclo.} $ \Z_{2}$-\text{extension}}}&
     &K^{(1)}_{\infty}(\sqrt{-t_2})\ar@{.}[uu]& \\
     K^{(1)} \ar@{-}[ur]\ar@{.}[uu] & & K^{(1)}_{\infty} \ar@{-}[ur]\ar@{.}[uu] & \\
     & F(\sqrt{-t_1})\ar@{-}[uu]_{G_1}& & F_{\infty}(\sqrt{-t_1})\ar@{-}[uu]_{G_1}\\
    F\ar@{-}[ur]\ar@{-}[uu]^{G_1}& & F_{\infty}\ar@{-}[ur]\ar@{-}[uu]^{G_1} \\
}  $$

\begin{proposition}\label{CohomologyEqual}
Let $ K = F(\sqrt{d_{1}}, \cdots, \sqrt{d_{r}}, \sqrt{-d}), $ where $ d_{1}, \cdots, d_{r}, d $
are positive odd square-free integers, $ F $ be a real abelian number field whose conductor is prime to
$ d \prod_{i=1}^{r} d_{i}. $ Suppose $ W_{K_{\infty}}(2) = \{\pm 1\} $ and
$ (E_{K_{\infty}}: \ E_{K^{+}_{\infty}} W_{K_{\infty}}) = 1. $ Then
        $$
         H^{i}(G_{r}, E_{K_{\infty}}) \cong  H^{i}(G_{r}, E_{K_{\infty}^{+}}) \qquad
\text{for} \ i = 1 \ \text{or} \ 2.
        $$
\end{proposition}

\begin{proof}  We have group isomorphism
$ E_{K_{\infty}} / E_{K_{\infty}^{+}} = E_{K_{\infty}^{+}} W_{K_{\infty}} / E_{K_{\infty}^{+}}
\cong W_{K_{\infty}} / E_{K_{\infty}^{+}} \cap W_{K_{\infty}}. $
As $ E_{K_{\infty}^{+}} \cap W_{K_{\infty}} = \{\pm 1\} = W_{K_{\infty}}(2), $ so
$ W_{K_{\infty}} /E_{K_{\infty}^{+}} \cap W_{K_{\infty}} = W_{K_{\infty}}^{\prime}, $ where
$ W_{K_{\infty}}^{\prime} = \{\zeta \in W_{K_{\infty}}: \zeta^{t} = 1 \ \text{for some odd number} \ t\}. $
Thus every element of $E_{K_{\infty}}/E_{K_{\infty}^+}$ is of odd order, so multiplication by $2$ is an
automorphism of $ E_{K_{\infty}} / E_{K_{\infty}^{+}}. $ Then it follows \cite[Proposition 1.6.2]{NSW} that
$ \widehat{H}^{i}(G_{r}, E_{K_{\infty}} / E_{K_{\infty}^{+}}) = 0 $ for all $ i \in \Z. $ On the other hand,
the short exact sequence of $ G_{r}$-modules
$$ 0 \rightarrow E_{K_{\infty}^{+}} \rightarrow E_{K_{\infty}} \rightarrow
E_{K_{\infty}}/E_{K^{+}_{\infty}} \rightarrow 0   $$
induces a long exact sequence of groups
$$  \cdots \rightarrow \widehat{H}^{i}(G_{r}, E_{K_{\infty}^{+}}) \rightarrow \widehat{H}^{i}(G_{r}, E_{K_{\infty}})
\rightarrow \widehat{H}^{i}(G_{r}, E_{K_{\infty}} /E_{K^{+}_{\infty}}) \rightarrow
\widehat{H}^{i+1}(G_{r}, E_{K_{\infty}^{+}}) \rightarrow \cdots.   $$
So from the above $ \widehat{H}^{i}(G_{r}, E_{K_{\infty}} / E_{K_{\infty}^{+}}) = 0, $ we get
$ \widehat{H}^{i}(G_{r}, E_{K_{\infty}}) \cong \widehat{H}^{i}(G_{r}, E_{K_{\infty}^{+}}) $ for all $ i \in \Z, $
and our conclusion follows because $ \widehat{H}^{i} = H^{i} (i = 1, 2) $ for finite groups
$ G_{r} (\cong \Z / 2 \Z $ here). The proof is completed.
\end{proof}

\begin{proposition} \label{CohomologyEqualPlusOne}
Let $ K = F(\sqrt{d_1},\cdots,\sqrt{d_r},\sqrt{-d})$, where $d_1,\cdots, d_r,d$ are positive odd square-free integers,
$F$ be a real abelian field whose conductor is prime to $d\prod_{i=1}^r d_i$. Suppose $\sqrt{2}\not\in K^+$ and
$\sqrt{d}\in K^+(\sqrt{2}). $ Then
    $$
    \chi(G_r,E_{K_{\infty}})=\chi(G_r,E_{K_{\infty}^+})+1.
    $$
\end{proposition}

\begin{proof} Consider the exact sequence of $ G_{r}$-modules
 $ 0 \rightarrow E_{K_{\infty}^{+}} \rightarrow E_{K_{\infty}} \rightarrow E_{K_{\infty}} / E_{K^{+}_{\infty}}
\rightarrow 0,  $ which induces the exact hexagon of groups and homomorphisms:
$$  \xymatrix{
& H^{1}(G_{r}, E_{K_{\infty}^{+}}) \ar@{->}[r] & H^{1}(G_{r}, E_{K_{\infty}})\ar@{->}[rd] & \\
H^{2}(G_{r}, E_{K_{\infty}} / E_{K_{\infty}^{+}}) \ar@{->}[ru]& & &H^{1}(G_{r}, E_{K_{\infty}} / E_{K_{\infty}^{+}})\ar@{->}[ld] \\
&H^{2}(G_{r}, E_{K_{\infty}})\ar@{->}[ul]^{\varphi}& H^{2}(G_{r}, E_{K_{\infty}^{+}})\ar@{->}[l] &
    }     $$
Next we compute $ H^{i}(G_{r}, E_{K_{\infty}} / E_{K_{\infty}^{+}}) \ (i = 1, 2). $ By Proposition
\ref{HasseUnitIndex} above, $ E_{K_{\infty}} = E_{K_{\infty}^{+}} W_{K_{\infty}}, $ so every element in
$ E_{K_{\infty}} / E_{K_{\infty}^{+}} $ is of the form $ \zeta E_{K_{\infty}^{+}} $ with $ \zeta\in W_{K_{\infty}}. $
We denote $$  \overline{W_{K_{\infty}}(2)} = \{\zeta E_{K_{\infty}^{+}} \in E_{K_{\infty}} / E_{K_{\infty}^{+}}: \
\zeta\in W_{K_{\infty}}(2)\},  $$
it is the subgroup of $ E_{K_{\infty}} / E_{K_{\infty}^{+}} $ generated by all elements of $ 2$-power orders.
Similarly, we denote
$$  \overline{W_{K_{\infty}}^{\prime}} = \{\zeta E_{K_{\infty}^{+}} \in E_{K_{\infty}} / E_{K_{\infty}^{+}}: \
\zeta \in W_{K_{\infty}}^{\prime}\},  $$
which is the subgroup of $ E_{K_{\infty}} / E_{K_{\infty}^{+}} $ generated by all elements of odd orders.
Since each element of $ E_{K_{\infty}} / E_{K_{\infty}^{+}} $ is either in $ \overline{W_{K_{\infty}}(2)} $ or
in $ \overline{W_{K_{\infty}}^{\prime}}, $ and $ \overline{W_{K_{\infty}}(2)} \cap \overline{W_{K_{\infty}}^{\prime}}
=\{1\}, $ we get $ E_{K_{\infty}} / E_{K_{\infty}^{+}} = \overline{W_{K_{\infty}}(2)} \times \overline{W_{K_{\infty}}^{\prime}}. $
Both $ \overline{W_{K_{\infty}}(2)} $ and $ \overline{W_{K_{\infty}}^{\prime}} $ are $ G_r$-modules, so by
\cite[Section I.2, Exercise 3, Page 24]{NSW}, we have
$$  \widehat{H}^{i}(G_{r}, E_{K_{\infty}} / E_{K_{\infty}^{+}}) =
\widehat{H}^{i}(G_{r}, \overline{W_{K_{\infty}}(2)}) \times \widehat{H}^{i}(G_{r}, \overline{W_{K_{\infty}}^{\prime}}). $$
Note that each element of $ \overline{W_{K_{\infty}}^{\prime}} $ is of odd order, so the map of multiplication
by $ 2 $ is an automorphism of $ \overline{W_{K_{\infty}}^{\prime}}, $ hence by \cite[Proposition 1.6.2]{NSW},
$ \widehat{H}^{i}(G_{r}, \overline{W_{K_{\infty}}^{\prime}}) = 0 $ for all $ i \in \Z. $ Therefore
$$  \widehat{H}^{i}(G_{r}, E_{K_{\infty}} / E_{K_{\infty}^{+}}) =
\widehat{H}^{i}(G_{r}, \overline{W_{K_{\infty}}(2)}) \quad \text{for all} \ i \in \Z. $$
Now we come to calculate $\widehat{H}^{i}(G_r,\overline{W_{K_{\infty}}(2)}). $ For the field $ K $ here, without loss of
generality, we may as well assume that there exists an odd prime number $ p $ such that $ p |d_{r} $ but
$ p \nmid d \prod_{i=1}^{r-1} d_{i}. $ Indeed, if otherwise, i.e., for a prime $ q, $ we have
$  q |d \prod_{i=1}^{r-1} d_{i} $ if $ q |d_{r}. $ Then fixes such a prime $ q, $ denote
$$ d^{\prime} = \begin{cases} d, & q \nmid d,  \\
{\rm sqf}(d d_{r}), & q | d,
\end{cases}
\quad \text{and} \quad d_{j}^{\prime} = \begin{cases} d_{j}, & q \nmid d_{j},  \\
{\rm sqf}(d_{j} d_{r}), & q | d_{j},
\end{cases} \quad (j = 1, \cdots, r - 1).
$$
We have $ q \nmid d^{\prime} \prod_{j=1}^{r-1} d_{j}^{\prime}, \
K = F(\sqrt{d_{1}^{\prime}}, \cdots, \sqrt{d_{r-1}^{\prime}}, \sqrt{d_{r}}, \sqrt{-d^{\prime}}), $
and $ \sqrt{d^{\prime}} \in K^{+}. $ \\
Thus for our assumption of $ K, $ we have $ \sqrt{d_{r}d} \not\in K^{(r-1)}(\sqrt{2}) =
F(\sqrt{2}, \sqrt{d_{1}}, \cdots, \sqrt{d_{r-1}}), $ then by Lemma \ref{RootsOfUnityinKInfty} above, we have $ W_{K_{\infty}(\sqrt{- d_{r}d})}(2)
= \{\pm 1\}, $ and so $ N_{G_{r}} W_{K_{\infty}}(2) \subseteq W_{K^{(r-1)}(\sqrt{d_{r}d})}(2) = \{\pm 1\} \subseteq E_{K_{\infty}^{+}}, $
hence $ N_{G_{r}} \overline{W_{K_{\infty}}(2)} = \{\overline{1}\} $ is trivial. To calculate
$ \widehat{H}^{0}(G_{r}, \overline{W_{K_{\infty}}(2)}), $ let $ \tau \in G_{r} \cong \Z / 2 \Z $ be the generator of $ G_{r}, $
and note that $$ \begin{aligned} (\overline{W_{K_{\infty}}(2)})^{G_{r}} &=
\{\zeta_{2^{t}} E_{K_{\infty}^{+}} \in \overline{W_{K_{\infty}}(2)} : \
\tau(\zeta_{2^{t}}) \zeta_{2^{t}}^{-1} \in E_{K_{\infty}^{+}}\} \\
&= \{\zeta_{2^{t}} E_{K_{\infty}^{+}} \in \overline{W_{K_{\infty}}(2)} : \
N_{G_{r}}(\zeta_{2^{n}}) \zeta_{2^{t}}^{-2} \in E_{K_{\infty}^{+}}\} \\
&= \{\zeta_{2^{t}} E_{K_{\infty}^{+}}
\in \overline{W_{K_{\infty}}(2)} : \ \zeta_{2^{t}}^{2} \in E_{K_{\infty}^{+}}\} =
\left<\sqrt{-1} \right> \cong \Z / 2 \Z,
\end{aligned}  $$
we have $$ H^{2}(G_{r}, \overline{W_{K_{\infty}}(2)}) \cong \widehat{H}^{0}(G_{r}, \overline{W_{K_{\infty}}(2)})
\cong (\overline{W_{K_{\infty}}(2)})^{G_{r}} / N_{G_{r}} \overline{W_{K_{\infty}}(2)} \cong \Z / 2 \Z. $$
For $ \widehat{H}^{-1}(G_{r}, \overline{W_{K_{\infty}}(2)}), $ consider the augmentation ideal
$ I_{G_{r}} = \left<\tau -1 \right>. $ For any $ \zeta_{2^{t}} E_{K_{\infty}^{+}} \in \overline{W_{K_{\infty}}(2)} $
with $ \zeta_{2^{t}} \in W_{K_{\infty}}(2), $ we have $ \zeta_{2^{t}}^{\tau - 1} =
N_{G_{r}}(\zeta_{2^{t}}) / \zeta_{2^{t}}^{2} = \pm1/\zeta_{2^{t}}^{2}, $ so $ (\zeta_{2^{t}}E_{K^{+}_{\infty}})^{\tau-1}
= (1 / \zeta^{2}_{2^{t}}) E_{K_{\infty}^+}. $ Also by Lemma \ref{RootsOfUnityinKInfty} above, we know that
$ {W_{K_{\infty}}(2)} $ contains all $ 2$-power roots of unity. Thus the equation $ x^{-2} = \alpha $ with
$ \alpha \in {W_{K_{\infty}}(2)} $ always has roots in $ {W_{K_{\infty}}(2)}. $ So, for each element
$ \zeta_{2^{t}} E_{K_{\infty}} $ of $ \overline{W_{K_{\infty}}(2)}, $ there exists an element
$ \zeta^{\prime} \in W_{K_{\infty}}(2) $ such that $ (\zeta^{\prime} E_{K_{\infty}})^{\tau -1} =
\zeta_{2^t} E_{K_{\infty}}. $ Hence $ I_{G_{r}} \overline{W_{K_{\infty}}(2)} =
\overline{W_{K_{\infty}}(2)}. $ Therefore
$$  \begin{aligned}
H^{1} (G_{r}, \overline{W_{K_{\infty}}(2)}) \cong \widehat{H}^{-1}(G_{r}, \overline{W_{K_{\infty}}(2)})
&\cong {\rm Ker}\left(N_{G_{r}}: \ \overline{W_{K_{\infty}}(2)} \to
\overline{W_{K_{\infty}}(2)}\right) / I_{G_{r}} \overline{W_{K_{\infty}}(2)} \\
&= \overline{W_{K_{\infty}}(2)} / \overline{W_{K_{\infty}}(2)} = 0.
\end{aligned}  $$
To sum up, we get $ \widehat{H}^{0}(G_{r}, E_{K_{\infty}} / E_{K_{\infty}^{+}}) \cong \Z / 2 \Z $ and
$ H^{1}(G_{r}, E_{K_{\infty}} / E_{K_{\infty}^{+}}) \cong 0. $  \\
As $ H^{2}(G_{r}, E_{K_{\infty}} / E_{K_{\infty}^{+}}) \cong
\widehat{H}^{0}(G_{r}, E_{K_{\infty}} / E_{K_{\infty}^{+}}) \cong \Z / 2 \Z, $ the homomorphism
$ \varphi $ in the above hexagon diagram is either $ 0 $ or surjective. If $ \varphi = 0, $ then
$ H^{2}(G_{r}, E_{K_{\infty}}) \cong H^{2}(G_{r}, E_{K^{+}_{\infty}}), $ and we have the exact sequence
$$  0 \rightarrow \Z /2 \Z \rightarrow H^{1}(G_{r}, E_{K_{\infty}^{+}}) \rightarrow H^{1}(G_{r}, E_{K_{\infty}})
\rightarrow 0.  $$
As  $ \# G_{r} = 2, $ we have $ 2 H^{1}(G_{r}, E_{K_{\infty}^+}) = 0 $ (see \cite[Proposition 1.6.1]{NSW}),
so $ H^{1}(G_{r}, E_{K_{\infty}^+})$ is an elementary $2$-group. From the above  exact sequence, we get
$$ \dim_{\F_{2}} H^{1}(G_{r}, E_{K_{\infty}^+}) = \dim_{\F_{2}} H^{1}(G_{r}, E_{K_{\infty}}) + 1.  $$
 i.e., $  \chi (G_{r}, E_{K_{\infty}}) = \chi (G_{r}, E_{K_{\infty}^{+}}) + 1. $ If $ \varphi$ is surjective,
then $ H^{1}(G_{r}, E_{K_{\infty}}) \cong H^{1}(G_{r}, E_{K^+_{\infty}}), $
and we have the exact sequence
$$   0 \rightarrow H^{2}(G_{r}, E_{K_{\infty}^{+}}) \rightarrow H^{2}(G_{r}, E_{K_{\infty}}) \rightarrow
 \Z /2 \Z \rightarrow 0,  $$  from which we can similarly get
$  \chi (G_{r}, E_{K_{\infty}}) = \chi (G_{r}, E_{K_{\infty}^{+}}) + 1. $
The proof is completed.
\end{proof}

\begin{lemma}\label{LambdaTwoForddividesProddi}
Let $ K = F(\sqrt{d_{1}}, \cdots, \sqrt{d_{r}}, \sqrt{-d}), $ where $ d_{1}, \cdots, d_{r} $
are square-free odd positive integers, $ F $ is a real abelian number field with conductor
$ \mathfrak{f}_{F/\Q} $ prime to $ d \prod_{i=1}^{r} d_{i}. $ Assume $ \sqrt{2} \notin K $
and $ \sqrt{d} \notin K^{+}(\sqrt{2}). $ \ Write
$ L:= F(\sqrt{d_{1}}, \cdots, \sqrt{d_{r}}, \sqrt{d}, \sqrt{-1}). $ Then if $ d \mid \prod_{i=1}^{r} d_{i}, $
we have $$  2 \lambda_{2}(K) = \lambda_{2}(L) - 1 + 2 \lambda_{2}(K^{+}) - \lambda_{2}(L^{+}).  $$
In particular, if Greenberg's conjecture holds, then $ 2 \lambda_{2}(K) = \lambda_{2}(L) - 1. $
\end{lemma}

\begin{proof} As discussed before, let
$ G = \text{\rm Gal}(L^{+} / F(\sqrt{d_{1}}, \cdots, \sqrt{d_{r}})) \cong \text{\rm Gal}(L_{\infty} /K_{\infty}) $
and $ K = L^{(r)}(\sqrt{d \cdot(-1)}). $ By Proposition \ref{CohomologyEqualPlusOne} above,
$  \chi (G, E_{L_{\infty}}) = \chi (G, E_{L_{\infty}^{+}}) + 1. $  \\
Also, for $ L_{\infty} / K_{\infty}, $ by Theorem \ref{IwasawaRiemannHurewitz} above, \\
$ \lambda_{2}(L) = 2 \lambda_{2}(K) + \sharp \{v \nmid 2 : v \ \text{is a place of} \ K_{\infty} \ \text{ramified in} \
L_{\infty} /K_{\infty} \} +  \chi (G, E_{L_{\infty}}). $ \ Since $ K(\sqrt{-1})/ K $ is unramified outside $ 2, $
and any $ \Z_{2}$-extension is unramified outside $ 2, $
hence $ L_{\infty} / K_{\infty} $ is unramified outside $ 2. $ Therefore
$$  \lambda_{2}(L) = 2 \lambda_{2}(K) + \chi (G, E_{L_{\infty}^{+}}) + 1. $$
Moreover, for $ L_{\infty}^{+} /F_{\infty}(\sqrt{d_{1}}, \cdots, \sqrt{d_{r}}), $ by Theorem \ref{IwasawaRiemannHurewitz}
above, we have $$ \begin{aligned} \lambda_{2}(L^{+}) = &2 \lambda_{2}(F(\sqrt{d_{1}}, \cdots, \sqrt{d_{r}})) +
\chi (G, E_{L_{\infty}^{+}}) + \\
&\sharp \{v \nmid 2 : v \ \text{is a place of} \
F_{\infty}(\sqrt{d_{1}}, \cdots, \sqrt{d_{r}}) \ \text{ramified in} \
L_{\infty}^{+} /F_{\infty}(\sqrt{d_{1}}, \cdots, \sqrt{d_{r}}) \}.
\end{aligned}  $$
Since $ d \mid \prod_{i=1}^{r} d_{i}, $ by Lemma \ref{Lemma 2.10}, we know that $ L^{+} / F(\sqrt{d_{1}}, \cdots, \sqrt{d_{r}}) $
is unramified outside $ 2. $ Hence \
$  \lambda_{2}(L^{+}) = 2 \lambda_{2}(F(\sqrt{d_{1}}, \cdots, \sqrt{d_{r}})) +
\chi (G, E_{L_{\infty}^{+}}), $ and so \
$  2 \lambda_{2}(K) = \lambda_{2}(L) - 1 + 2 \lambda_{2}(F(\sqrt{d_{1}}, \cdots, \sqrt{d_{r}}))
- \lambda_{2}(L^{+}).  $  \ The proof is completed.
\end{proof}

\begin{theorem}\label{MainTheorem}
Let $ d_{1}, \cdots, d_{r}, d $ be square-free positive integers, $ F $
be a real abelian number field with conductor $ \mathfrak{f}_{F/\Q} $ prime to
$ d \cdot \prod_{i=1}^{r} d_{i}. $ Let $ K = F(\sqrt{d_{1}}, \cdots, \sqrt{d_{r}}, \sqrt{-d}) $ be a degree
$ 2^{r+1} $ extension over $ F. $ Assume $ \sqrt{2} \notin K. $ Let $ K_{\infty} $ be the
cyclotomic $ \Z_{2}$-extension of $ K. $
Denote $ K^{(i)} = F(\sqrt{d_{1}}, \cdots, \sqrt{d_{i}}) $ and  $ t_{i} = d_{i} \cdots d_{r} d \ (1 \leq i \leq r). $
Assume there exists an odd prime $ p_{r} $ such that $ p_{r} \mid d_{r} $ but $ p_{r} \nmid d \prod_{i=1}^{r-1} d_{i}. $
For any positive $ n, $ denote $ \text{\rm sqf}(n) = \prod_{\text{prime} \ p|n, 2 \nmid \nu_{p}(n)} p. $
Also, for $ i \in \{1, \cdots, r\}, $ denote  \\
$  s_{i} = \sharp \{v : v \ \text{is a place of} \ K_{\infty}^{(i-1)}(\sqrt{-t_{i}}) \ \text{over a
prime number} \ p \ \text{with} \ p \mid (d_{i}, \text{\rm sqf}(t_{i+1})) \ \text{and} \ p \nmid
2 \prod_{t=1}^{i-1} d_{t} \}, $ \\ and \ $ f_{i} = \sharp \{v : v \ \text{is a place of} \
K_{\infty}^{(i-1)} \ \text{over a prime number} \ p \ \text{with} \ p \mid d_{i} \ \text{and} \ p \nmid
2 \prod_{t=1}^{i-1} d_{t} \}. $  Then
$$ \lambda_{2}(K) = \lambda_{2}(K(\sqrt{2})) =
2^{r} \lambda_{2}(F(\sqrt{-t_{1}})) + \lambda_{2}(K^{+}) - 2^{r} \lambda_{2}(F)
+ \sum_{i=1}^{r} 2^{r-i} (s_{i} - f_{i}) + \delta,  $$
where $ \delta = \left \{\begin{array}{l} 1, \quad \quad \text{if} \ \sqrt{d} \in K^{+}(\sqrt{2}),  \\
0, \quad \quad  \text{if otherwise}.
\end{array} \right. $
\end{theorem}

\begin{proof} Since $ K $ and $ K(\sqrt{2}) $ has the same cyclotomic $ \Z_{2}$-extension, it can be easily
seen that all the quantities $ s_{i}, f_{i} \ (i = 1, \cdots, r) $ and $ \delta $ are independent of
the $ 2$-factor of $ d \prod_{i=1}^{r} d_{i}. $ So we may as well assume that
$ 2 \nmid d_{1} \cdot \cdots d_{r} \cdot d. $ By \cite[Chapter II, Exercise 1.2.5]{Gra},
$$ \begin{aligned}  &s_{i} = \sharp \left\{v : v \nmid 2 \ \text{is a place of} \ K_{\infty}^{(i-1)}(\sqrt{-t_{i}}) \
 \text{ramified in} \ K_{\infty}^{(i)}(\sqrt{-t_{i+1}}) \right\}, \\
&f_{i} = \sharp \left\{v : v \nmid 2 \ \text{is a place of} \ K_{\infty}^{(i-1)} \ \text{ramified in} \
K_{\infty}^{(i)} \right\}.
\end{aligned}  $$
We divide our discuss in three cases as follows:
\begin{enumerate}[$a)$]
\item  There exists an odd prime number $ p $ such that $ p \mid d $ and $ p \nmid \prod_{i=1}^{r} d_{i}; $
\item  $ \sqrt{d} \in K^{+}(\sqrt{2});  $
\item  $ \sqrt{d} \notin K^{+}(\sqrt{2}) $ and $ d \mid \prod_{i=1}^{r} d_{i}. $
\end{enumerate}
For cases a) and b), we use induction on $ r. $ If $ r = 1, $ then $ K = F(\sqrt{d_{1}}, \sqrt{-d}) $
and $ t_{1} = d_{1} d. $ For $ K_{\infty} / F(\sqrt{-t_{1}}) $ and $ K_{\infty}^{+} / F_{\infty}, $
by Theorem \ref{IwasawaRiemannHurewitz} above, we have
$  \lambda_{2}(K) = 2 \lambda_{2} (F(\sqrt{-t_{1}})) + s_{1} + \chi (G_{1}, E_{K_\infty}) $
and $ \lambda_{2}(K^{+}) = 2 \lambda_{2}(F) + f_{1} + \chi (G_{1}, E_{K_{\infty}^{+}}). $
Also by Propositions \ref{CohomologyEqual} and \ref{CohomologyEqualPlusOne} above, we have
$$  \chi (G_{1}, E_{K_{\infty}}) = \begin{cases}
        \chi (G_{1}, E_{K_{\infty}^{+}}), & \text{case} \ a) \\
        \chi (G_{1}, E_{K_{\infty}^{+}}) + 1. & \text{case} \ b)
    \end{cases}   $$
Therefore, $$ \lambda_{2}(K) = \begin{cases}
        2 \lambda_{2} (F(\sqrt{-t_{1}})) + \lambda_{2}(K^{+}) - 2 \lambda_{2}(F) + (s_{1} - f_{1}),& \text{case} \ a) \\
       2 \lambda_{2} (F(\sqrt{-t_{1}})) + \lambda_{2}(K^{+}) - 2 \lambda_{2}(F) + (s_{1} - f_{1})
+ 1. & \text{case} \ b)
    \end{cases}   $$ So our conclusion holds for $ r = 1. $  \\
Now assume our conclusion holds for $ r = n-1, $ and we consider $ r = n. $ As discussed before, for
$ K_{\infty} / K_{\infty}^{(n-1)}(\sqrt{-t_{n}}) $ and $ K_{\infty}^{+} / F_{\infty}^{(n-1)}, $
by Theorem \ref{IwasawaRiemannHurewitz} above, we have
$$  \lambda_{2}(K) = 2 \lambda_{2}(K^{(n-1)}(\sqrt{-t_{n}})) + s_{n} + \chi(G_{n}, E_{K_\infty}) \
\text{and} \ \lambda_{2}(K^{+}) = 2 \lambda_{2}(K^{(n-1)}) + f_{n} + \chi(G_{n}, E_{K_{\infty}^{+}}). $$
Also by Proposition \ref{CohomologyEqual} and \ref{CohomologyEqualPlusOne} above,
$ \chi(G_{n}, E_{K_\infty}) = \chi(G_{n}, E_{K_\infty}^{+}) + \delta. $
So  $$  \lambda_{2}(K) = 2 \lambda_{2}(K^{(n-1)}(\sqrt{-t_{n}})) + \lambda_{2}(K^{+}) -
2 \lambda_{2}(K^{(n-1)}) + (s_{n} - f_{n}) + \delta,  $$
where $ \delta =  \begin{cases} 0, & \text{case} \ a) \\
1. & \text{case} \ b)
\end{cases}  $  \\
For $  \lambda_{2}(K^{(n-1)}(\sqrt{-t_{n}})), $ notice that, as assumed, there exists an odd prime number
$ p $ such that $ p \mid t_{n} $ and $ p \nmid \prod_{i=1}^{n-1} d_{i}. $ So by induction,
$$  \lambda_{2}(K^{(n-1)}(\sqrt{-t_{n}})) = 2^{r-1} \lambda_{2}(F(\sqrt{-t_{1}})) +
\lambda_{2}(K^{(n-1)}) - 2^{n-1} \lambda_{2}(F) + \sum_{i=1}^{n-1} 2^{n-1-i} (s_{i} - f_{i}). $$
Therefore,  $$  \lambda_{2}(K) =
    \begin{cases}
        2^{n} \lambda_{2}(F(\sqrt{-t_{1}})) + \lambda_{2}(K^{+})  - 2^{n} \lambda_{2}(F) +
\sum_{i=1}^{n} 2^{n-i} (s_{i} - f_{i}), & \text{case} \ a) \\
        2^{n} \lambda_{2}(F(\sqrt{-t_{1}})) + \lambda_{2}(K^{+})  - 2^{n} \lambda_{2}(F) +
\sum_{i=1}^{n} 2^{n-i} (s_{i} - f_{i}) + 1. & \text{case} \ b)
    \end{cases}   $$
That is, the conclusion holds for $ r = n. $ Hence by inductive principle, our conclusion
holds for cases a) and b).   \\
Lastly, we consider case c). Note that $ \sqrt{2} \notin K(\sqrt{d}), $ so by
Lemma \ref{LambdaTwoForddividesProddi} above, we have
$ 2 \lambda_{2}(K) = \lambda_{2}(L) - 1 + 2\lambda_{2}(K^{+}) - \lambda_{2}(L^{+}), $
where
$$   \begin{aligned}  L= K(\sqrt{-1}) &= F(\sqrt{d_{1}}, \cdots, \sqrt{d_{r}}, \sqrt{d}, \sqrt{-1})  \\
&= F(\sqrt{d_{1}^{\prime}}, \cdots, \sqrt{d_{r+1}^{\prime}}, \sqrt{-1}),
\end{aligned} \quad \ \text{and} \
d_{i}^{\prime} =\begin{cases}
        d_{1}, & 1 \leq i \leq r-1,  \\
        d,     &  i = r, \\
        d_{r}, & i + r + 1.
    \end{cases}   $$
Denote $ t_{i}^{\prime} = d_{i}^{\prime} \cdots d_{r+1}^{\prime} \ (1 \leq i \leq r+1) $
and $ t_{r+2}^{\prime} = 1. $ From the discussion of case b) above, we know  \\
 $$  \begin{aligned} \lambda_{2}(L) &= 2^{r+1} \lambda_{2} (F(\sqrt{-t_{1}^{\prime}})) +
\lambda_{2}(F(\sqrt{d_{1}^{\prime}}, \cdots, \sqrt{d_{r+1}^{\prime}})) - \\
&2^{r+1} \lambda_{2}(F) + \sum_{i=1}^{r+1} 2^{r+1-i} (s_{i}^{\prime} - f_{i}^{\prime}) + 1 \\
&= 2^{r+1} \lambda_{2}(F(\sqrt{-t_{1}})) + \lambda_{2}(L^{+}) - 2^{r+1} \lambda_{2}(F) +
\sum_{i=1}^{r+1} 2^{r+1-i} (s_{i}^{\prime} - f_{i}^{\prime}) + 1, \end{aligned}  $$
where $ s_{i}^{\prime} = \sharp \{v : v \nmid 2 \ \text{is a place in} \
L_{\infty}^{(i-1)}(\sqrt{-t_{i}^{\prime}}) \ \text{ramifying in} \
L_{\infty}^{(i)}(\sqrt{-t_{i+1}^{\prime}})\}, $ \\
$  f_{i}^{\prime} = \sharp \{v : v \nmid 2 \ \text{is a place in} \
L_{\infty}^{(i-1)} \ \text{ramifying in} \ L_{\infty}^{(i)} \}. $  So
$$  \lambda_{2}(K) = 2^{r} \lambda_{2}(F(\sqrt{-t_{1}})) + \lambda_{2}(K^{+}) - 2^{r} \lambda_{2}(F) +
\sum_{i=1}^{r+1} 2^{r-i} (s_{i}^{\prime} - f_{i}^{\prime}). $$
Since $ L_{\infty}^{(r+1)}(\sqrt{-t_{r+2}^{\prime}}) = L_{\infty}^{(r)}(\sqrt{-t_{r+1}^{\prime}})
(\sqrt{-1}), $ we have $ s_{r+1}^{\prime} = 0. $ For $ 1 \leq i \leq r-1, $ notice that
$ L^{(i)} = F(\sqrt{d_{1}^{\prime}}, \cdots, \sqrt{d_{}^{\prime}}) = K^{(i)}$, and $ d \mid
\prod_{i=1}^{r} d_{i}. $ So \\
$ s_{i}^{\prime} = s_{i}, \ f_{i}^{\prime} = f_{i} \ (1 \leq i \leq r-1), $  \\
$ s_{r}^{\prime} = \sharp \{v : v \ \text{is a prime of} \
K_{\infty}^{(r-1)}(\sqrt{-t_{r}}) \ \text{over a prime number} \ p \ \text{with} \ p \mid d \
\text{and} \ p \nmid \prod_{i=1}^{r-1} d_{i} \}, $  \\
$ f_{r}^{\prime} = \sharp \{v : v \ \text{is a prime of} \
K_{\infty}^{(r-1)} \ \text{over a prime number} \ p \ \text{with} \ p \mid d \ \text{and} \ p \nmid
 \prod_{i=1}^{r-1} d_{i} \}, $  \\
$ f_{r+1}^{\prime} = \sharp \{v : v \ \text{is a prime of} \
L_{\infty}^{(r)} \ \text{over a prime number} \ p \ \text{with} \ p \mid d_{r} \ \text{and} \ p \nmid
d \prod_{i=1}^{r-1} d_{i} \}. $  \\
On the other hand, $$  \begin{aligned} s_{r} &= \sharp \{v : v \ \text{is a prime of} \
K_{\infty}^{(r-1)}(\sqrt{-t_{r}}) \ \text{over a prime number} \ p \ \text{with} \ p \mid d \
\text{and} \ p \nmid \prod_{i=1}^{r-1} d_{i} \} \\
&= s_{r}^{\prime};  \\
f_{r} &= \sharp \{v : v \ \text{is a prime of} \ K_{\infty}^{(r-1)} \ \text{over a prime number} \
p \ \text{with} \ p \mid d_{r} \ \text{and} \ p \nmid \prod_{i=1}^{r-1} d_{i} \}  \\
&= f_{r}^{\prime} + \sharp \{v : v \ \text{is a prime of} \ K_{\infty}^{(r-1)} \ \text{over a prime number} \
p \ \text{with} \ p \mid d_{r} \ \text{and} \ p \nmid d \prod_{i=1}^{r-1} d_{i} \}.
 \end{aligned}    $$
Notice that, for $ p > 2, $ if $ p \mid d_{r} $ and $ p \nmid d \prod_{i=1}^{r-1} d_{i}, $
then by \cite[Chapter II, Exercise 1.2.5]{Gra}, we know that all the primes over $ p $ split completely in
$ L^{(r)} / K^{(r-1)} = K^{(r-1)}(\sqrt{d}) / K^{(r-1)}. $  Hence  $ f_{r} = f_{r}^{\prime} +
\frac{1}{2} f_{r+1}^{\prime}. $ To sum up, we obtain
$$  \lambda_{2}(K) = 2^{r} \lambda_{2}(F(\sqrt{-t_{1}})) + \lambda_{2}(K^{+}) - 2^{r} \lambda_{2}(F)
+ \sum_{i=1}^{r} 2^{r-i} (s_{i} - f_{i}). $$
The proof is completed.
\end{proof}

\begin{corollary}\label{ComputeSiFi}
Let $ K = \Q(\sqrt{d_{1}}, \cdots, \sqrt{d_{r}}, \sqrt{-d}), $  where
$ d_{1}, \cdots, d_{r}, d $ are square-free positive integers. Assume $ [K : \Q] = 2^{r+1} $
and there exists a prime $ p_{r} > 2 $ such that $ p_{r} \mid d_{r} $ but $ p_{r} \nmid d \prod_{i=1}^{r-1} d_{i}. $
Then for any integer $ 1 \leq i\leq  r, $ we have
$$  s_{i} = \sum_{\substack{\text{primes} \ p | (d_{i}, \text{\rm sqf}(t_{i+1})), \\ p \nmid
2 \prod_{t=1}^{i-1} d_{t}}} 2^{\nu_{2}(p^{2} - 1) + i - 3} \quad \text{and} \quad f_{i} =
\sum_{\substack{\text{primes} \ p | d_{i}, \\ p \nmid 2 \prod_{t=1}^{i-1} d_{t}}} 2^{\nu_{2}(p^{2} - 1) + i - 4}. $$
\end{corollary}

\begin{proof} For any integer $ 1 \leq i\leq  r, $ by definition, we have   \\
$  f_{i} = \sharp \{v : v \ \text{is a prime of} \ K_{\infty}^{(i-1)} \ \text{over a prime number} \ p \
\text{with} \ p \mid d_{i} \ \text{and} \ p \nmid \prod_{t=1}^{i-1} d_{t} \}. $ By \cite[Chapter II, Exercise 1.2.5]{Gra},
such places $ v $ split completely in
$ K_{\infty}^{(i-1)} / \Q. $ So by Lemma \ref{DecompositionsOfOddPrimeInQn} above, $$ f_{i} =
\sum_{\text{primes} \ p | d_{i}, p \nmid 2 \prod_{t=1}^{i-1} d_{t}} 2^{\nu_{2}(p^{2} - 1) + i - 4}. $$
The value $ s_{i} $ can be similarly obtained.
The proof is completed.
\end{proof}

\begin{corollary}\label{MainCorollary}
Let $ K = \Q(\sqrt{d_{1}}, \cdots, \sqrt{d_{r}}, \sqrt{-d}) $ with
$ [K : \Q] = 2^{r+1} \ (r \geq 0), $  where $ d_{1}, \cdots, d_{r}, d $ are square-free positive integers. Then
$$  \lambda_{2}(K) = \lambda_{2}(K^{+}) + \sum_{\substack{\text{primes} \ p | \prod_{i=1}^{r} d_i,  \\ p \neq 2}}
2^{\nu_{2}(p^{2} - 1) + r - \theta - 4} + \sum_{\substack{\text{primes} \ p | d, \\ p\nmid 2 \prod_{i=1}^{r} d_{i}}}
2^{\nu_{2}(p^{2} - 1) + r - \theta - 3} - 2^{r - \theta} + \delta,  $$ where
$ \delta = \left \{\begin{array}{l} 1, \quad \quad \text{if} \ \sqrt{d} \in K^{+}(\sqrt{2}),  \\
0, \quad \quad  \text{if otherwise}
\end{array} \right. $ and $ \theta = \left \{\begin{array}{l} 1, \quad \quad \text{if} \ \sqrt{2} \in K^{+},  \\
0, \quad \quad  \text{if otherwise}.
\end{array} \right. $
\end{corollary}

\begin{proof} By Theorem \ref{FerreroImaQuaLambda} above, we only need to consider the case $ r \geq 1. $
If $ \sqrt{2} \notin K^{+}, $ we may as well assume there exists a prime $ p_{r} > 2 $ such that
$ p_{r} \mid d_{r} $ but $ p_{r} \nmid d \prod_{i=1}^{r-1} d_{i}. $
Then by Theorems \ref{FerreroImaQuaLambda}, \ref{MainTheorem} and
Corollary \ref{ComputeSiFi} above, we know that
$$  \begin{aligned}
    &\lambda_{2}(K) = \sum_{\substack{\text{primes} \ p \mid \text{\rm sqf}(d \prod_{i=1}^{r} d_{i}), \\ p \neq 2}}
2^{\nu_{2}(p^{2} - 1) - 3 + r} - 2^{r}  \\
&+ \sum_{i=1}^{r} \left(\sum_{\substack{\text{primes} \ p | (d_{i}, \text{\rm sqf}(t_{i+1})), \\
p \nmid 2 \prod_{t=1}^{i-1} d_{t}}} 2^{\nu_{2}(p^{2} - 1) +r - 3} -
\sum_{\substack{\text{primes} \ p | d_{i}, \\ p \nmid 2 \prod_{t=1}^{i-1} d_{t}}}
2^{\nu_{2}(p^{2} - 1) + r - 4} \right) + \delta.
    \end{aligned}    $$
While $$  \begin{aligned}  &\sum_{i=1}^{r} \sum_{\substack{\text{primes} \ p | d_{i}, \\
p \nmid 2 \prod_{t=1}^{i-1} d_{t}}} 2^{\nu_{2}(p^{2} - 1) + r - 4} =
\sum_{\substack{\text{primes} \ p \mid \prod_{i=1}^{r} d_{i}, \\ p \neq 2}}
2^{\nu_{2}(p^{2} - 1) + r - 4},     \\
&\sum_{i=1}^{r} \sum_{\substack{\text{primes} \ p | (d_{i}, \text{\rm sqf}(t_{i+1})), \\
p \nmid 2 \prod_{t=1}^{i-1} d_{t}}} 2^{\nu_{2}(p^{2} - 1) +r - 3} =
\sum_{\substack{\text{primes} \ p | \prod_{i=1}^{r} d_{i}, \\
p \nmid \text{\rm sqf}(\prod_{i=1}^{r} d_{i}), \\ p \neq 2}} 2^{\nu_{2}(p^{2} - 1) + r - 3},   \\
&\sum_{\substack{\text{primes} \ p | \text{\rm sqf}(d \prod_{i=1}^{r} d_{i}), \\ p \neq 2}}
2^{\nu_{2}(p^{2} - 1) + r - 3} = \sum_{\substack{\text{primes} \ p | \prod_{i=1}^{r} d_{i},
p | \text{\rm sqf}(d \prod_{i=1}^{r} d_{i}), \\ p \neq 2}} 2^{\nu_{2}(p^{2} - 1) + r - 3} +  \\
&\sum_{\substack{\text{primes} \ p | d, \\ p \nmid 2 \prod_{i=1}^{r} d_{i}}}
2^{\nu_{2}(p^{2} - 1) + r - 3}.
\end{aligned}  $$
If $ \sqrt{2} \in K, $  for $ r = 1, $ we have $ K = \Q(\sqrt{2}, \sqrt{-d}), $ then $ \lambda_{2}(K) =
\lambda_{2}(\Q(\sqrt{-d})), $ and the conclusion follows from Theorem \ref{FerreroImaQuaLambda} above. For $ r \geq 2, $
we may as well assume $ d_{1} = 2, $ then $ \lambda_{2}(K) =
\lambda_{2}(\Q(\sqrt{d_{2}}, \cdots, \sqrt{d_{r}}, \sqrt{-d})). $ Now $ \sqrt{2} \notin
\Q(\sqrt{d_{2}}, \cdots, \sqrt{d_{r}}, \sqrt{-d}) $ and $ r - 1 \geq 1, $ so by the above discussion,
our conclusion holds. The proof is completed.
\end{proof}

\begin{example} Let $ K = \Q(\sqrt{d_{1}}, \sqrt{-d}) $ be a bi-quadratic  number
field, where $ d_{1} $ and $ d $ are positive square-free odd integers. Assume Greenberg's conjecture
holds for $ \Q(\sqrt{d_{1}}), $ then by Corollary \ref{MainCorollary} above, we have
$$  \lambda_{2}(K) = \sum_{\text{primes} \ p | d_{1}}
2^{\nu_{2} (p^{2} - 1) - 3} + \sum_{\text{primes} \ p | d, \ p \nmid d_{1}}
2^{\nu_{2}(p^{2} - 1) - 2} - 2 + \delta,  $$  where
$ \delta = \left \{\begin{array}{l} 1, \quad \quad \text{if} \ d = d_{1} \ \text{or} \ 1,  \\
0, \quad \quad  \text{if otherwise}.
\end{array} \right. $   \\
In particular, given a positive integer $ N, $ let $ K= \Q(\sqrt{\ell \ell^{\prime}}, \sqrt{-1}), $
where $ \ell $ and $ \ell^{\prime} $ are different prime numbers, $ \nu_{2}( \ell - 1) = N + 2, \
\ell^{\prime} \equiv 5 \ (\text{\rm mod} \ 8), $ and $ \left(\frac{\ell}{\ell^{\prime}} \right) = -1. $
By \cite[Theorem 1.1]{Mo3}, there exist infinitely many such fields $ K. $  \\
Since $ \nu_{2}(\ell^{2} - 1) = \nu_{2}(\ell - 1) + \nu_{2}((\ell - 1) + 2) = (N+2)+1 = N + 3 $
and $ \nu_{2}((\ell^{\prime})^2 - 1) = \nu_{2}(\ell^{\prime} - 1) + \nu_{2}(\ell^{\prime} + 1) = 2 + 1
=3. $ So $ \lambda_{2}(K) = 2^{N}, $ which is uniformly compatible with \cite[Theorem 1.1]{Mo3}, and
shows that there are infinitely many CM number fields, over which the unramified Iwasawa module
having any given $ \Z_{2}$-rank.
\end{example}

\begin{example} Let $ K = \Q(\sqrt{d_{1}}, \sqrt{d_{2}}, \sqrt{-d}) $ with $ [K : \Q] = 2^{3}, $
where $ d_{1}, d_{2} $ and $ d $ are positive square-free odd integers. Assume Greenberg's conjecture
holds for $ \Q(\sqrt{d_{1}}, \sqrt{d_{2}}), $ then by Corollary \ref{MainCorollary} above, we have
$$  \lambda_{2}(K) = \sum_{\text{primes} \ p | d_{1}d_{2}}
2^{\nu_{2} (p^{2} - 1) - 2} + \sum_{\text{primes} \ p | d, \ p \nmid d_{1}d_{2}}
2^{\nu_{2}(p^{2} - 1) - 1} - 4 + \delta,  $$  where
$ \delta = \left \{\begin{array}{l} 1, \quad \quad \text{if} \ d \in \ \{1, d_{1}, d_{2},
\text{\rm sqf}(d_{1} d_{2}) \},  \\
0, \quad \quad  \text{if otherwise}.
\end{array} \right. $   \\
 In particular, for $ K = \Q(\sqrt{q_{1}}, \sqrt{q_{2}}, \sqrt{-1}), $ where $ q_{1} $ and $ q_{2} $
 are two different odd prime numbers satisfying $ q_{1} \equiv 7 \ (\text{\rm mod} \ 16) $ and
 $ q_{2} \equiv 3 \ (\text{\rm mod} \ 8). $ Then it is not difficult to show that $ \lambda_{2}(K) = 3, $
which is uniformly compatible with \cite[Theorem 3]{CAZ}; Also, if the above $ q_{1} $ and $ q_{2} $
satisfying $ q_{1} \equiv q_{2} \equiv 3 \ (\text{\rm mod} \ 8), $ then
$ \lambda_{2}(K) = 1, $ which is uniformly compatible with \cite[Lemma 3.5]{ACZ1}.
\end{example}

\section{ Iwasawa invariants and class number parity of multi-quadratic number fields }  

\begin{lemma}\label{IwasawaInvariantVanishFirstLevelOddCl}
Let $ K $ be a number field, $ p $ be a prime number. For an
$ \Z_{p}$-extension $ K_{\infty} / K, $ if all the prime ideals over $ p $
are totally ramified in $ K_{\infty} / K, $ then
$$ \lambda_{p}(K_{\infty}) = \mu_{p}(K_{\infty}) = \nu_{p}(K_{\infty}) = 0 \Leftrightarrow
A_{K_{1}} = 0,  $$
where $ A_{K_{1}} $ is the Sylow $ p$-subgroup of the ideal class group of $ K_{1}. $
\end{lemma}

\begin{proof} By assumption, we have $ \sharp A_{K_{n}} \leq \sharp A_{K_{n+1}} $ for
all integers $ n \geq 1. $ So $ \lambda_{p}(K_{\infty}) = \mu_{p}(K_{\infty}) =
\nu_{p}(K_{\infty}) = 0 \Leftrightarrow A_{K_{n}} = 0 $ for all integers $ n \geq 1. $
Since $ \sharp A_{K} \leq \sharp A_{K_{1}}, $ our conclusion follows directly from \cite[Theorem 1]{Fu}.
The proof is completed.
\end{proof}

Let $ F / \Q $ be an abelian $ 2$-extension, $ F_{\infty}
= \bigcup _{n=1}^{\infty} F_{n} $ be the cyclotomic $ \Z_{2}$-extension with
$ [F_{n} : F ] = 2^{n}, $ and let $ A_{F_{n}} $ denote the Sylow $ 2$-subgroup of the ideal class group
of $ F_{n}. $ Assume $ 8 \nmid \mathfrak{f}_{F/\Q}, $ the conductor of $ F, $ then $ \sqrt{2} \notin F, $
we have $ F_{n} = F \Q_{n} $ because $ F \bigcap \Q_{\infty} = \Q. $ Then all the primes over $ 2 $
are totally ramified in $ F_{\infty} / F. $ Thus $ \lambda_{2}(F) = \mu_{2}(F) = \nu_{2}(F)
= 0 \Leftrightarrow A_{F_{1}} = 0, $ i.e., equivalently, the class number of $ F(\sqrt{2}) $
is odd. In particularly, if $ F = \Q(\sqrt{d_{1}}, \cdots, \sqrt{d_{r}}) $ with $ [F : \Q] =
2^{r}, $ where $ d_{1}, \cdots, d_{r} $ are square-free positive integers. Then $ 8 \nmid
\mathfrak{f}_{F/\Q} \Leftrightarrow 2 \nmid \prod_{i=1}^{r} d_{i}. $  \\
Now for any abelian $ 2$-extension $ K / \Q, $ denote by
$ K^{\text{Elm}} $ the maximal sub-extension of $ K / \Q $ such that the Galois group
$ \text{\rm Gal}(K^{\text{Elm}} / \Q) $ is an elementary $ 2$-group, in other words,
$ K^{\text{Elm}} $ is the subfield of $ K $ fixed by the Frattini subgroup of
$ \text{\rm Gal}(K / \Q). $ Let $ K^{\text{Gen}} $ be the genus field of $ K / \Q, $ i.e.,
the maximal unramified abelian extension of $ K $ such that $ K^{\text{Gen}} / \Q $
is also abelian. Also, let $ K^{\text{NGen}} $ be the narrow genus field of $ K / \Q, $
i.e., the maximal abelian extension of $ K $ over which all the finite places are
unramified and such that $ K^{\text{NGen}} / \Q $ is also abelian. In particularly,
if $ K $ is totally real, then $ K^{\text{Gen}} = K^{\text{NGen}} \bigcap \R. $

\begin{theorem}[\cite{Ya}, Theorem 2.5]\label{YamamotoIwaVanishTwo}
Let $ F/\Q $ be a real abelian $ 2$-extension
with $ 8 \nmid \mathfrak{f}_{F/\Q}. $ Then all the Iwasawa invariants in the cyclotomic
$ \Z_{2}$-extension vanish if and only if $ F = F^{\text{\rm Gen}} $ and $ F^{\text{\rm Elm}} $ is
one of the following fields:
\begin{enumerate}[$a)$]				
\item $ F^{\text{\rm Elm}} = \Q(\sqrt{p}), $ where $ p $ is a prime number satisfying
$$   \begin{aligned}
 p \equiv 3 \ (\text{\rm mod} \ 4), \quad & \text{or}  \\
  p \equiv 5 \ (\text{\rm mod} \ 8), \quad & \text{or} \\
  p \equiv 1 \ (\text{\rm mod} \ 8), \quad & \text{and} \ \left(\frac{2}{p}\right)_{4} \left(\frac{p}{2}\right)_{4} = -1.
        \end{aligned}    $$
\item $  F^{\text{\rm Elm}} = \Q(\sqrt{pq}), $ where $ p $ and $ q $ are two different prime numbers satisfying
$$   p \equiv 3 \ (\text{\rm mod} \ 4) \quad \text{and} \quad q \equiv 3 \ (\text{\rm mod} \ 8).   $$
\item $ F^{\text{\rm Elm}} =\Q(\sqrt{p},\sqrt{q}), $ where $ p $ and $ q $ are two different prime numbers satisfying
$$       \begin{aligned}
            &p \equiv q \equiv 3 \ (\text{\rm mod} \ 8), & \text{or} \\
            &p \equiv 3 \ (\text{\rm mod} \ 8), \  q \equiv 5 \ (\text{\rm mod} \ 8), & \text{or} \\
            &p \equiv 3 \ (\text{\rm mod} \ 8), \  q \equiv 7 \ (\text{\rm mod} \ 8), & \text{or} \\
            &p \equiv 5 \ (\text{\rm mod} \ 8), \  q \equiv 7 \ (\text{\rm mod} \ 8), & \text{or} \\
            &p \equiv 5 \ (\text{\rm mod} \ 8) \ q \equiv 1 \ (\text{\rm mod} \ 8), \left(\frac{q}{p}\right)=-1,
            \left(\frac{2}{q}\right)_4\left(\frac{q}{2}\right)_4=-1, & \text{or} \\
            &p\equiv q \equiv 5 \ (\text{\rm mod} \ 8), \left(\frac{q}{p}\right)=1, \left(\frac{q}{p}\right)_{4}
            \left(\frac{p}{q}\right)_{4}=-1, & \text{or} \\
            &p\equiv q \equiv 5 \ (\text{\rm mod} \ 8), \left(\frac{q}{p}\right)=-1, \left(\frac{2q}{p}\right)_{4}
            \left(\frac{2p}{q}\right)_{4} \left(\frac{pq}{2}\right)_{4} = 1. &
        \end{aligned}  $$
\item $  F^{\text{\rm Elm}} = \Q(\sqrt{p}, \sqrt{q}, \sqrt{\ell}), $ where $ p, q $ and $ \ell $ are different prime numbers satisfying
$$  \begin{aligned}
            &p \equiv q \equiv 3, \ \ell \equiv 5 \ (\text{\rm mod} \ 8), \left(\frac{pq}{\ell}\right)=-1, & \text{or} \\
            &p \equiv q \equiv 3, \ \ell \equiv 7 \ (\text{\rm mod} \ 8), \left(\frac{pq}{\ell}\right)=-1, & \text{or} \\
            &p \equiv 3, \ q \equiv 5, \ \ell \equiv 7 \ (\text{\rm mod} \ 8), \left(\frac{q}{\ell}\right) = -1.&
        \end{aligned} $$
\item $  F^{\text{\rm Elm}} = \Q(\sqrt{pq}, \sqrt{\ell}), $  where $ p, q $ and $ \ell $ are different prime numbers satisfying
$$   \begin{aligned}
        &p \equiv q \equiv 3, \ \ell \equiv 5 \ (\text{\rm mod} \ 8), \left(\frac{pq}{\ell}\right) = -1, & \text{or} \\
        &p \equiv 3, q \equiv 7, \ell \equiv 5 \ (\text{\rm mod} \ 8), \left(\frac{\ell}{q}\right) = -1.
        \end{aligned}   $$
\item $ F^{\text{\rm Elm}} = \Q(\sqrt{pq}, \sqrt{p \ell}), $ where $ p, q $ and $ \ell $ are different prime numbers satisfying
$    p \equiv q \equiv 3, \ \ell \equiv 7 \ (\text{\rm mod} \ 8), \left(\frac{pq}{\ell}\right)=-1.  $
\end{enumerate}
\end{theorem}

\begin{theorem}[\cite{Zh}, Corollary 1]\label{ZhangXiankeNGenusField}
Let $ F = \Q(\sqrt{d_{1}}, \cdots, \sqrt{d_{r}}) $ with $ [F : \Q] = 2^{r}, $
where $ d_{i} \in \Z \ (1 \leq i \leq r) $ are square-free integers. Then
$$   F^{\text{\rm NGen}} = F(\sqrt{p_{1}^{\ast}}, \cdots, \sqrt{p_{m}^{\ast}}), $$
where $ p_{1}, \cdots, p_{m} $ are all the odd prime numbers ramifying in $ F, $ and $ p_{i}^{\ast}
=(-1)^{(p-1)/2} p_{i} \ (1 \leq i \leq m). $
\end{theorem}

\begin{corollary}\label{TotallyRealMultQuaIwasawaInaVanish}
\begin{enumerate}[$1)$]				
\item Let $ F $ be a real multi-quadratic number field with $ 8 \nmid \mathfrak{f}_{F/\Q}. $
Then all the Iwasawa invariants vanish in the cyclotomic $ \Z_{2}$-extension of $ F $ if and only if
$ F =  F^{\text{\rm Elm}} $ as in Theorem \ref{YamamotoIwaVanishTwo} above. \\
\item Let $ L $ be a real multi-quadratic number field with $ \sqrt{2} \in L. $ Then $ 2 \nmid h(L) $
if and only if $ L = F (\sqrt{2}) $  for some field $ F $ as in 1) above.
\end{enumerate}
\end{corollary}

\begin{proof} 1) \ Since $ F =  F^{\text{\rm Elm}}, $ by Lemma \ref{IwasawaInvariantVanishFirstLevelOddCl} above,
$ \lambda_{2}(F) = \mu_{2}(f) = \nu_{2} (F) = 0  \Leftrightarrow
F = F^{\text{Gen}} =  F^{\text{Elm}} $ as described in Theorem \ref{YamamotoIwaVanishTwo}. Therefore, we only need to
show that all of the fields $ F $ as in Theorem \ref{YamamotoIwaVanishTwo} satisfying $ F^{\text{Gen}} = F. $ \\
To see this, let $ F = \Q(\sqrt{p_{1}}, \cdots, \sqrt{p_{m}}), $ where $ p_{1}, \cdots, p_{m} $
are distinct odd primes, and all of them are ramified in $ F /\Q. $ We may as well assume that
$ p_{1} \equiv \cdots \equiv p_{s} \equiv 1 \ (\text{\rm mod} \ 4) $ and
$ p_{s+1} \equiv \cdots \equiv p_{m} \equiv 3 \ (\text{\rm mod} \ 4) $ for some integer $ s : 1 \leq s \leq m. $
Then by Theorem \ref{ZhangXiankeNGenusField} above, we have
$$  \begin{aligned} F^{\text{NGen}} &= F(\sqrt{p_{1}}, \cdots, \sqrt{p_{s}}, \sqrt{-p_{s+1}}, \cdots, \sqrt{-p_{m}}) \\
&= F(\sqrt{p_{1}}, \cdots, \sqrt{p_{s}}, \sqrt{p_{s+1}p_{m}}, \cdots, \sqrt{p_{m-1}p_{m}}, \sqrt{-p_{m}}).
  \end{aligned} $$
So $ F^{\text{Gen}} = F(\sqrt{p_{1}}, \cdots, \sqrt{p_{s}}, \sqrt{p_{s+1}p_{m}}, \cdots, \sqrt{p_{m-1}p_{m}}). $
Now let $ F $ be as in cases a)-f) of Theorem \ref{YamamotoIwaVanishTwo}, then it can be easily verified that $ F^{\text{Gen}} = F, $
and we are done.  \\
2) \ By assumption, there exists a subfield $ F $ of $ L $ such that $ 8 \nmid \mathfrak{f}_{F/\Q}. $
Since every place over $ 2 $ is totally ramified in $ F_{\infty} / F, $ by Lemma \ref{IwasawaInvariantVanishFirstLevelOddCl},
we have $ A_{L} = 0 \Leftrightarrow \lambda_{2}(F) = \mu_{2}(F) = \nu_{2}(F) = 0. $ So by 1) above, we are done.
The proof is completed.
\end{proof}

\begin{theorem}\label{MainApplicationParity}
Let $ K $ be an imaginary multi-quadratic number field with $ \sqrt{2} \in K. $
Then $ 2 \nmid h(K) $ if and only if $ K $ is taken as one of the following form
\begin{enumerate}[$1)$]
\item $ \Q(\sqrt{2}, \sqrt{-p}), $ where $ p $ is a prime number with $ p \equiv 3 \ (\text{\rm mod} \ 8); $ \\
\item $ \Q(\sqrt{2}, \sqrt{-1}, \sqrt{-p}), $ where $ p $ is a prime number with $ p \equiv 3, 5 \ (\text{\rm mod} \ 8); $  \\
\item $ \Q(\sqrt{2}, \sqrt{-p}, \sqrt{-q}), $ where $ p $ and $ q $ are different prime numbers with
$ p, q \equiv 3 \ (\text{\rm mod} \ 8); $  \\
\item $ \Q(\sqrt{2}, \sqrt{-1}). $
\end{enumerate}
Moreover, for $ K = \Q(\sqrt{2}, \sqrt{-p}), $ if $ p \equiv 5 \ (\text{\rm mod} \ 8), $ then $ 2 \mid h(K) $
but $ 4 \nmid h(K). $
\end{theorem}

\begin{proof} $ \Rightarrow. $ By assumption, $ K = \Q(\sqrt{2}, \sqrt{d_{1}}, \cdots, \sqrt{d_{r}}, \sqrt{-d}) $
with $ [K : \Q] = 2^{r+2} $ for some positive odd square-free integers $ d_{1}, \cdots, d_{r}, d. $
We may as well assume that there is an odd prime number $ p \mid d_{r} $ but
$ p \nmid d \prod_{i=1}^{r-1} d_{i}. $ Write $ F = \Q(\sqrt{d_{1}}, \cdots, \sqrt{d_{r}}) $ and
$ E = F (\sqrt{-d}). $ Then by Lemma \ref{IwasawaInvariantVanishFirstLevelOddCl} above,
the class number $ h(K) $ is odd if and only if all the
Iwasawa invariants $ \lambda_{2}, \mu_{2} $ and $ \nu_{2} $ of the cyclotomic
$ \Z_{2}$-extension of $ E $ are equal to zero. Now assume that $ h(K) $ is odd. Then $ \lambda_{2}(E) = 0. $
So the intersection of $ K $ with the Hilbert $ 2$-class field of $ F(\sqrt{2}) $ is equal to $ F $
because $ K $ is totally imaginary and $ F(\sqrt{2}) $ is totally real. Hence $ A_{F(\sqrt{2})} $
is isomorphic to a quotient of $ A_{K}. $ So if $ A_{K} = 0, $ then $ A_{F(\sqrt{2})} = 0, $
which is equivalent to say that all the Iwasawa invariants of the cyclotomic $ \Z_{2}$-extension of
$ F $ are vanishing (via the Lemma \ref{IwasawaInvariantVanishFirstLevelOddCl} above). While Corollary
\ref{TotallyRealMultQuaIwasawaInaVanish} above provides all the possibility of
such $ F. $ Notice that for any odd prime number $ \ell, $ we have $ \nu_{2}(\ell^{2} - 1) \geq 3, $ and
the equality holds if and only if $ \ell \equiv 3, 5 \ (\text{\rm mod} \ 8). $  \\
If $ r \geq 2, $ then by Corollary \ref{MainCorollary} above, we have
$$ \lambda_{2}(E) = \sum_{\text{primes} \ p \mid \prod_{i=1}^{r} d_{i}, p>2} 2^{\nu_{2}(p^{2} - 1) +r -4} +
\sum_{\text{primes} \  p \mid d, p \nmid 2\prod_{i=1}^{r} d_{i}} 2^{\nu_{2}(p^{2} - 1) + r -3} -2^{r} + \delta. $$
Since $ \lambda_{2}(E) = 0, $ by modulo $ 2 $ of the two sides of the above equality, we get $ \delta = 0. $
So $ d \notin F(\sqrt{2}). $ Hence $ d_{1}d_{2} = pq $ and $ d \mid pq $ with $ p, q \equiv 3, 5 \ (\text{\rm mod} \ 8), $
also $ r = 2 $  because any prime number $ \ell $ satisfying $ v_{2} (\ell^{2} - 1) \geq 3. $
Then by Corollary \ref{TotallyRealMultQuaIwasawaInaVanish} above, we have $ F = \Q(\sqrt{p}, \sqrt{q}), $ but
$ d \mid pq $ and $ d \notin F(\sqrt{2}), $
a contradiction! Therefore, $ r = 0 $ or $ 1. $  \\
If $ r = 0, $ i.e., $ E = \Q(\sqrt{-d}). $ Then by Theorem \ref{FerreroImaQuaLambda} above, we have $ \lambda_{2}(E) = 0, $
which implies that $ d = p $ is a prime number satisfying $ p \equiv 3, 5 \ (\text{\rm mod} \ 8) $
or $ d = 1. $ \\
If $ r = 1, $ then $ F = \Q(\sqrt{p}) $ or $ \Q(\sqrt{pq}) $ for some (different) odd prime numbers $ p $
and $ q. $ For $ F = \Q(\sqrt{pq}), $ by Corollary \ref{MainCorollary} above, we have
$$ \lambda_{2}(E) = \sum_{\text{primes} \ \ell \mid pq} 2^{\nu_{2}(\ell^{2} - 1) -3} +
\sum_{\text{primes} \  \ell \mid d, \ell \nmid pq} 2^{\nu_{2}(\ell^{2} - 1) -2} -2 + \delta. $$
Since $ \lambda_{2}(E) = 0, $ we have $ p, q \equiv 3 \ (\text{\rm mod} \ 8) $ and $ \delta = 0, $ also
every prime factor of $ d $ must divide $ pq. $ While $ \delta = 0 \Leftrightarrow
\sqrt{d} \notin \Q(\sqrt{2}, \sqrt{pq}). $ Hence $ d = p $ or $ q. $ Similarly, for $ F = \Q(\sqrt{p}), $
we have $ \delta = 1 $ because $ \lambda_{2}(E) = 0. $ So $ d = p $ or $ 1 $ because $ \delta = 1
\Leftrightarrow \sqrt{d} \in \Q(\sqrt{2}, \sqrt{p}). $  \\
$ \Leftarrow. $ Conversely, we assume that $ K $ is one of the given cases. Notice that,
by \cite[Theorem 5.4]{ACZ2}, $ h(K) $ is odd for $ K $ in case 2). \\
For case 1), if $  p \equiv 3 \ (\text{\rm mod} \ 8), $ then $ 2 $ is totally ramified in $ \Q(\sqrt{2}, \sqrt{p}) $
and inert in $ \Q(\sqrt{-p}), $ so all the places over $ 2 $ ramify in
$ \Q(\sqrt{2}, \sqrt{-1}, \sqrt{p}) / \Q(\sqrt{2}, \sqrt{-p}). $ Then the Hilbert $ 2$-class group of
$ \Q(\sqrt{2}, \sqrt{-p}) $ is isomorphic to a quotient of the Hilbert $ 2$-class group of
$ \Q(\sqrt{2}, \sqrt{-1}, \sqrt{p}). $ As the class number of $ \Q(\sqrt{2}, \sqrt{-1}, \sqrt{-p}) $ is odd as
in case 2) above, so is the class number of $ \Q(\sqrt{2}, \sqrt{-p}). $ If $  p \equiv 5 \ (\text{\rm mod} \ 8), $
then it is easy to see that $ \Q(\sqrt{2}, \sqrt{-1}, \sqrt{p}) / \Q(\sqrt{2}, \sqrt{-p}) $ is unramified.
Also as the class number of $ \Q(\sqrt{2}, \sqrt{-1}, \sqrt{p}) $ is odd as in case 2) above,
it can be easily seen that the class number of $ \Q(\sqrt{2}, \sqrt{-p}) $ is even but not divided by $ 4. $  \\
For case 3), denote $ K = \Q(\sqrt{2}, \sqrt{-p}, \sqrt{-q}) = \Q(\sqrt{2}, \sqrt{pq}, \sqrt{-q}),
k =  \Q(\sqrt{- q}),  d = pq. $ We may as well assume that $ q > 3. $
Since $ p \equiv q \equiv 3 \ (\text{\rm mod} \ 8), $ by Corollary \ref{DecompositionsOfOddPrimeInQn} above,
we know that both $ p $ and $ q $ are not decomposable in $ \Q_{\infty} / \Q. $ By \cite[Theorem 3.1]{MR},
we have $$  \sharp (E_{k_{1}} / E_{k_{1}} \cap N_{k_{1}(\sqrt{d}) / k_{1}}(k_{1}(\sqrt{d})^{\text{Gen}})) =
\sharp (E_{\Q_{1}} / E_{\Q_{1}} \cap N_{\Q_{1}(\sqrt{d}) / \Q_{1}}(\Q_{1}(\sqrt{d})^{\text{Gen}})). $$
By Theorem \ref{ZhangXiankeNGenusField} above,
$$ \begin{aligned} k_{1} (\sqrt{d})^{\text{Gen}} = \Q (\sqrt{-q}, \sqrt{2}, \sqrt{d})^{\text{Gen}}
= \Q (\sqrt{-q}, \sqrt{2}) (\sqrt{-q}, \sqrt{-p}) = \Q (\sqrt{-q}, \sqrt{2}, \sqrt{-p}) = K.  \\
\Q_{1}(\sqrt{d})^{\text{Gen}} = \Q (\sqrt{2}, \sqrt{d}) (\sqrt{-p}, \sqrt{-q}) \cap \R = \Q (\sqrt{2}, \sqrt{d})
= K \cap \R.
\end{aligned}  $$
We know that the class number $ h(\Q(\sqrt{2})) = 1 $ and $ 2 \nmid h(\Q(\sqrt{2}, \sqrt{d})).  $ moreover,
both $ p $ and $ q $ are inert in $ \Q (\sqrt{2}), $ and $ \Q (\sqrt{2}, \sqrt{d}) / \Q (\sqrt{2}) $ is unramified
outside $ \{p, q \}, $ by \cite[Proposition 1]{EM}, we have
$$ \text{rank}_{2}(E_{\Q_{1}} / E_{\Q_{1}} \cap N_{\Q_{1}(\sqrt{d}) / \Q_{1}}(\Q_{1}(\sqrt{d}))) =
- \text{rank}_{2} ( \text{Cl}_{\Q (\sqrt{2}, \sqrt{d})}) + 2 - 1 = 1.  $$
Also $ E_{\Q_{1}}^{2} \subset E_{\Q_{1}} \cap N_{\Q_{1}(\sqrt{d}) / \Q_{1}}(\Q_{1}(\sqrt{d})), $
so $ E_{\Q_{1}} / E_{\Q_{1}} \cap N_{\Q_{1}(\sqrt{d}) / \Q_{1}}(\Q_{1}(\sqrt{d})) $ is an elementary
$ 2$-abelian group, hence,
$$ \sharp (E_{\Q_{1}} / E_{\Q_{1}} \cap N_{\Q_{1}(\sqrt{d}) / \Q_{1}}(\Q_{1}(\sqrt{d})^{\text{Gen}})) = 2
= \sharp (E_{k_{1}} / E_{k_{1}} \cap N_{k_{1}(\sqrt{d}) / k_{1}}(k_{1}(\sqrt{d})^{\text{Gen}})). $$
Therefore, by \cite[Proposition 1]{EM} again, we get $ \text{rank}_{2} (\text{Cl}_{K}) = 0. $ This proves
case 3). The proof is completed.
\end{proof}

\begin{remark}
    For $ K = \Q(\sqrt{2}, \sqrt{-p}, \sqrt{-q}) $ in case 3) of Theorem \ref{MainApplicationParity} above,
we find the result of \cite[Corollary 3.5]{CCH} stated that $ 2 \mid h(K), $ which is different from ours.
However, \\
$ h(\Q(\sqrt{2}, \sqrt{-11}, \sqrt{33})) = 1 $ (see \cite{Y}, \cite{F}), which
is consistent with our conclusion.
\end{remark}

{ \bf Acknowledgments } \ We would like to thank the referee
for helpful suggestions and comments.

\end{document}